\documentclass{amsart}

\newcommand{\acknowledgement}{\subsection*{Acknowledgements}}

\usepackage{amsmath}
\usepackage{amssymb}
\usepackage{amsrefs}
\usepackage{tikz}
\usepackage{tikz-3dplot}
\usetikzlibrary{positioning}
\usepackage{color}
\usepackage{paralist}
\usepackage{mathtools}
\usepackage{colonequals}
\usepackage{breqn}

\usepackage{graphicx}
\usepackage{caption}
\usepackage{verbatim}
\usepackage[foot]{amsaddr}

\usepackage{thm-restate}

\numberwithin{equation}{section} 
\numberwithin{figure}{section} 
\theoremstyle{plain}

\newtheorem{thm}{Theorem}[section]
\newtheorem*{thm*}{Theorem}

\newtheorem{cor}[thm]{Corollary}
\newtheorem*{cor*}{Corollary}

\newtheorem{prop}[thm]{Proposition}
\newtheorem*{prop*}{Proposition}

\newtheorem{lem}[thm]{Lemma}
\newtheorem*{lem*}{Lemma}
\theoremstyle{definition}

\newtheorem{defn}[thm]{Definition}
\newtheorem*{defn*}{Definition}
\newtheorem{example}[thm]{Example}
\newtheorem*{example*}{Example}
\theoremstyle{remark}

\newtheorem{rem}[thm]{Remark}
\newtheorem*{rem*}{Remark}

\theoremstyle{plain}

\newcommand{\Id}{\operatorname{Id}}
\newcommand{\R}{\mathbb{R}}
\newcommand{\Z}{\mathbb{Z}}
\newcommand{\C}{\mathbb{C}}

\newcommand{\e}{\operatorname{e}}

\newcommand{\ceq}{\coloneqq}
\newcommand{\F}{\mathcal{F}}

\newcommand{\defin}[1]{\textbf{#1}}

\newcommand{\supp}{\operatorname{supp}}

\begin{document}


\title[Coisotropic Hofer-Zehnder capacities]{Coisotropic Hofer-Zehnder capacities and
non-squeezing for relative embeddings}

\author{Samuel Lisi}
\thanks{SL: partially 
supported by the ERC Starting Grant of Fr\'ed\'eric Bourgeois StG-239781-ContactMath and by the ERC Starting Grant of Vincent Colin Geodycon.}

\address{Samuel Lisi\\University of Mississippi, P.O. Box 1848, University, MS 38677, USA}
\email{stlisi@olemiss.edu}
\author{Antonio Rieser}
\thanks{AR: supported by the ISF grant 723/10, ERC 2012 Advanced Grant 20120216, Short Visit Grants from the Contact and Symplectic Topology Research Network of the European Science Foundation, and Cat\'{e}dras CONACYT / 1076.}

\address{Antonio Rieser\\CONACYT-CIMAT, Centro de Investigaci\'on en Matem\'aticas, Jalisco s/n, Col. Valenciana, CP 36023 Guanajuato, GTO, Mexico}
\email{antonio.rieser@cimat.mx}


\begin{abstract}
We introduce the notion of a symplectic capacity relative to a co\-isotropic submanifold of a symplectic manifold, and we construct two examples of such capacities through modifications of the Hofer-Zehnder capacity. As a consequence, we obtain a non-squeezing theorem for symplectic embeddings relative to co\-isotropic constraints and existence results for leafwise chords on energy surfaces. 
\end{abstract}

\maketitle

\section{Introduction}

Symplectic capacities are an important tool in the study of symplectic rigidity phenomena. 
The first one was constructed
by Gromov \cite{Gromov_1985}, 
and the notion was subsequently axiomatized by Ekeland and Hofer  \cite{Ekeland_Hofer_1989}.
Many such capacities exist; indeed,
each phenomenon of symplectic rigidity arguably gives rise to its own capacity.
A large number of examples are described in \cite{Cieliebak_Hofer_Latschev_Schlenk_2007}, and relationships between them, notably energy-capacity inequalities, lead to interesting connections between Hamiltonian dynamics and symplectic topology.


Very little is known about capacities defined relative to special
submanifolds $N \hookrightarrow M$ of a symplectic manifold, and
even in the Lagrangian case there are many open questions.  Barraud
and Cornea introduced the first relative capacities for the Lagrangian
case, the Lagrangian Gromov width and relative packing numbers
\cite{Barraud_Cornea_2007}.  Upper bounds for the Lagrangian Gromov
width of the Clifford torus in $\C P^n$ were computed by Biran and
Cornea \cite{Biran_Cornea_2009}, and Buhovsky \cite{Buhovsky_2010}
later computed lower bounds for the Clifford torus. Schlenk's
constructions \cite{Schlenk_2005} also work in the relative case,
and therefore compute the relative packing numbers for $k\leq 6$
balls in $(\C P^{2},\R P^{2})$.  In \cite{Rieser_2014}, the second
author defined a blow-up and blow-down procedure for Lagrangian
submanifolds, and used it to compute the Lagrangian Gromov width
of a class of Lagrangians that are fixed point sets of real, rank-$1$
symplectic manifolds.  In \citelist{\cite{Zehmisch_2012a} \cite{Zehmisch_2012b}},
Zehmisch constructed a capacity of a manifold $(M,\omega)$ from
embeddings of $n$-disk bundles over a Lagrangian submanifold and
related it to the geometry of the Lagrangian. In \cite{Borman_McLean_2014},
Borman and McLean constructed a spectral capacity for wrapped Floer
homology, and used it to study the Lagrangian Gromov width of closed
Lagrangian submanifolds in Liouville manifolds.  Dimitroglou Rizell
\cite{Rizell_2013} gave examples of compact Lagrangians in $\C^3$
with infinite Barraud-Cornea Lagrangian width, building on
\cite{Ekholm_etal_2013}.

In this paper, we study a notion of a capacity relative to a coisotropic
submanifold, which we call a coisotropic capacity. In a heuristic
sense, if a symplectic capacity measures the `width' of a symplectic
manifold, a coisotropic capacity similarly measures the symplectic
`size' of a coisotropic embedding inside an ambient symplectic manifold.
We construct a family of such capacities, analogous to the Hofer-Zehnder
capacity, indexed by a suitable equivalence relation on the coisotropic
submanifolds.

    We recall that a coisotropic submanifold is foliated by the leaves of the
    characteristic foliation.  A Hamiltonian trajectory that starts and ends on
    the same leaf of this foliation is called a leafwise chord. 
    As an application of the capacity we introduce, 
    we obtain existence of leafwise chords for the coisotropic submanifold for
    almost every energy level of an autonomous Hamiltonian, under the
    assumptions of having a finite capacity neighbourhood and transversality of
    the level set to the coisotropic submanifold. (See Theorem \ref{thm:Almost
    existence}.)

    Leafwise chords have been studied extensively in the literature, perhaps
    first appearing in the work of Moser \cite{Moser_1978}. We mention a few
    works that similarly approach this problem from an
    energy--capacity--inequality point of view, notably Hofer \cite{Hofer_1990},
    Ginzburg \cite{Ginzburg_2007}, Dragnev \cite{Dragnev_2008}, Ziltener
    \cite{Ziltener_2010}, G\"urel \cite{Gurel_2010}, Albers and Frauenfelder
    \cite{Albers_Frauenfelder_2010_TandA}, Albers and Momin \cite{Albers_Momin_2010}, 
    Usher \cite{Usher_2011} and Kang \cite{Kang_2013}.
    
    Of particular relevance to us are \cite{Ginzburg_2007}*{Theorem 2.7} and
    \cite{Gurel_2010}*{Theorem 1.1}. 
    These papers show that in the case of 
    coisotropic submanifolds of restricted contact type, 
    there is a lower bound on the
    leafwise displacement energy of the coisotropic submanifold coming from the
    symplectic area of a disk tangent to a leaf of the characteristic foliation.

\begin{defn} \label{def:W} Let
\begin{enumerate}
\item $\R^{n,k} \coloneqq \left\{ x \in \R^{2n} \vert x = (x_1, \dots, x_n, y_1, \dots, y_k, 0, \dots, 0) \right\}.$
\item $W(R) \coloneqq \left \{ (x_1, \dots, x_n, y_1, \dots, y_n) \in \R^{2n} \; | \; x_n^2 + y_n^2  < R^2 \text{ or } y_n < 0 
\right\}$
\item $W^{n,k}(R) \coloneqq  W(R) \cap \R^{n,k}$
\item $B(a,r)$ is the open ball of radius $r$ centered at
\[a \ceq (0,\dots,0,b_n) \in \R^{2n},\]
and $B(r)$ is the open ball of radius $r$ centered at the origin.
\item $B^{n,k}(r) \coloneqq B(r) \cap \R^{n,k}$
\end{enumerate}
\end{defn}

\begin{defn}
\label{def:Basic defs}
Let $(M, \omega)$ be a symplectic manifold and let $N \subset M$ be a
submanifold. Then, $N$ is \emph{coisotropic} if the symplectic orthogonal 
$TN^\omega \subset TN$.

The restriction $\omega|_N$ defines a distribution on $N$, consisting of
the kernel of $\omega|_N$.  By the Frobenius integrability theorem, this
distribution is integrable and integrates to the \textit{characteristic
foliation}.  The leaves of this distribution are the
\textit{isotropic leaves}.
\end{defn}

\begin{example}
\label{ex:Isotropic leaves of Rnk}
Let $\omega_0$ denote the standard symplectic form on $\R^{2n}$, and recall  
that $\R^{n,k}$ is the linear coisotropic subspace of $(\R^{2n},\omega_0)$ consisting
of the first $n+k$ coordinates, i.e. 
\begin{equation*}
\R^{n,k} = \left \lbrace x \in \R^{2n} \; \vert  \; 
x = (x_1,\dots,x_n,y_1,\dots, y_k,0,\dots,0) \right \rbrace.
\end{equation*}
Let $V_0$ be the linear subspace
\begin{equation*}
V_0 = \left \lbrace x \in \R^{2n} \; 
\vert \; x = (0,\dots,0,x_{k+1},\dots,x_n,0,\dots,0) \right \rbrace, \\
\end{equation*}
and note that any leaf $F$ in the characteristic foliation $\mathcal{F}$ of $\R^{n,k}$ has the form $z + V_0$, for some $z \in \R^{n,k}$.
\end{example}

\begin{defn}
A \textit{coisotropic equivalence relation} on $N$ is an equivalence
relation $\sim$ with the property that if $x$, $y$ are in the same
isotropic leaf, then $x \sim y$.

In particular, the \textit{leaf relation}, given by $x \sim y$ if and only if $x,
y$ are in the same isotropic leaf, is the finest coisotropic
equivalence relation. 
The \textit{trivial relation} defined by $x \sim y$ for
every pair $x, y \in N$ is the coarsest coisotropic relation.
\end{defn}

Note that if $N$ is a connected Lagrangian, 
there is only one coisotropic equivalence relation since there is only one leaf.

\begin{defn} \label{def:relative_embedding}
    Let $(M_0, \omega_0)$ and $(M_1, \omega_1)$ be symplectic manifolds and let
    $N_0, N_1$ be coisotropic submanifolds of $M_0, M_1$ respectively. 
    Let $\sim_0$ and $\sim_1$ be coisotropic equivalence relations on $N_0,
    N_1$.

    Then, an embedding $\psi \colon M_0 \to M_1$ is a \textit{relative
    symplectic embedding}, 
    \[
        \psi \colon (M_0, N_0, \omega_0) \hookrightarrow (M_1, N_1,
        \omega_1)
    \]
    if $\psi^*\omega_1 = \omega_0$ and $\psi^{-1}(N_1) = N_0$.

    The embedding $\psi$ \textit{respects the pair of coisotropic equivalence relations $(\sim_0,\sim_1)$} if
    furthermore, for every $x, y \in C$, if $\psi(x) \sim_1 \psi(y)$ then $x \sim_0 y$.

    If $\psi \colon (M_0, N_0, \omega_0) \hookrightarrow (M_1, N_1, \omega_1)$
    is a relative embedding, we define the pull-back relation $\sim_\psi$ on
    $N_0$ by 
    \[
        x \sim_\psi y \Longleftrightarrow \psi(x) \sim_1 \psi(y).
    \]
    Thus, $\psi$ respects the pair $(\sim_0, \sim_1)$
    if $\sim_0$ is a coarser relation than the pull-back $\sim_\psi$. 

\end{defn}
    In particular, if $N_0, N_1$ are Lagrangian, this recovers the definition of
    a relative symplectic embedding, first introduced (without the terminology) in Barraud-Cornea \cite{Barraud_Cornea_2007}, Section 1.3.3, and formally defined in Biran-Cornea \cite{biran_cornea_2008}, Section 6.2.
Observe also that $\psi$ respects the relations $\sim_\psi$ and $\sim_1$ by
construction of the pull-back. 
If $\sim$, $\approx$ are two equivalence relations on the coistropic submanifold
$N$, the identity on $(M, N, \omega)$ respects $\sim, \approx$ if and only if
$\sim$ is coarser than $\approx$.
\begin{example} \label{ex:construction_from_ambient}
    The first class of non-trivial examples comes from considering a properly
    embedded coisotropic submanifold $C$ in an ambient symplectic manifold, say
    $\R^{2n}$. We now consider all pairs $(U, N)$ so that there exists an
    embedding $\psi \colon U \to \R^{2n}$ for which $\psi(N) = C \cap \psi(U)$.
    Then, we may take the coisotropic equivalence relation on $N$ to be the
    pull-back of the leaf relation on $C$ by $\psi$. 
    Described more concretely, we say $x \sim y$ for $x, y \in N$ if $\psi(x)$
    and $\psi(y)$ are in the same leaf of $C$. 

\end{example}

\begin{defn}
A \emph{coisotropic capacity} is a map $(M,N,\omega, \sim) \mapsto 
c(M,N,\omega, \sim)$ which
associates to a tuple $(M,N,\omega, \sim)$ consisting of a symplectic manifold 
$(M,\omega)$, a coisotropic submanifold $N^{n+k}\hookrightarrow M$, $k<n$, and 
a coisotropic equivalence relation $\sim$, a non-negative number or infinity 
and satisfies the
following axioms:

\begin{enumerate}
\item \label{Axiom:Monotonicity}Monotonicity. If there exists a relative
symplectic embedding, respecting the coisotropic equivalence relations $\sim_0,
\sim_1$ on $N_0, N_1$
\begin{equation*}
\phi \colon (M_0,N_0,\omega_0)\hookrightarrow (M_1,N_1,\omega_1)
\end{equation*} for $M_0$ and $M_1$ of the same dimension, 
then \[c(M_0,N_0,\omega_0, \sim_0) \leq c(M_1,N_1,\omega_1, \sim_1).\]
\item \label{Axiom:Conformality}Conformality.  For fixed $(M, N, \omega, \sim)$, 
$$c(M,N,\alpha\omega, \sim)=|\alpha|c(M,N,\omega, \sim), \alpha \in \mathbb{R}\backslash \{0\}.$$

\item \label{Axiom:Nontriviality}Non-triviality. With $\sim$ denoting the leaf
    relation (see Definition \ref{def:Basic defs}), 
    $c\left(B(1),B^{n,k}(1),\omega_0, \sim \right) =
    c\left(W(1),W^{n,k}(1),\omega_0, \sim \right) = \pi/2$, where $W(1)$, $W^{n,k}(1)$
    are as in Definition \ref{def:W}.

\end{enumerate}

In general, a coisotropic capacity may not be defined for all possible tuples,
but only for a distinguished class. In particular, most of our examples will be
of this nature.
\end{defn}

\begin{rem}
When the symplectic form and the equivalence relation $\sim$ in $(M,N,\omega)$
are understood, we will abbreviate this to $(M,N)$. 
\end{rem}

\begin{rem}
The non-triviality axiom is subtly different from the one required for a symplectic capacity (as in 
\cite{Hofer_Zehnder_1994}). 
Let $Z(1) = B^2(1) \times \C^{n-1}$ be the standard symplectic cylinder, and let  $Z^{n,k}(1) \coloneqq Z(1) 
\cap \R^{n,k}$. For a symplectic capacity $c_0$, the non-triviality axiom is $c_0(B(1)) = c_0(Z(1) ) = \pi$,
and rules out the volume. The non-triviality axiom for a coisotropic capacity 
serves to rule out taking a standard symplectic capacity: 
for any standard symplectic capacity $c_0$, $c_{0}( W(1) )$ is infinite.
If we replaced this axiom with a weaker one, for instance requiring instead
\[
c( Z(1), Z^{n,k}(1)) = \frac{\pi}{2},
\]
we would be able to take a standard symplectic capacity $c_0(M, \omega)$ and define 
$c(M, N, \omega) = \frac{1}{2} c(M, \omega)$. 

Observe also that by considering embeddings of the form $\Id \times \psi$ where $\psi \colon \R^2 \to \R^2$ is symplectic,
we may construct an embedding of $Z(1)$ to $W(1)$ so that $Z^{n,k}(1)$ is mapped to $W^{n,k}(1)$, and thus
the weaker condition is implied by the stronger one.

The point of the non-triviality condition \ref{Axiom:Nontriviality} is thus to
rule out the trivial examples of rescaled symplectic capacities, but also
implies the weaker (perhaps more natural seeming) non-triviality condition.

\end{rem}

The coisotropic capacities that we will introduce are constructed similarly to the Hofer-Zehnder capacity, 
and depend on several classes of Hamiltonians that we define below.

\begin{defn} \label{def:simple}
An autonomous Hamiltonian  $H \colon M \to \R$ is \it{simple} if 
\begin{enumerate}
\item There exists a compact set $K\subset M$ (depending on $H$) and a constant $m(H)$ such that
$K\subset M \setminus \partial M$, $\emptyset \neq K\cap N \subsetneq N$,
and
\begin{equation*}
H(M \setminus K)=m(H).
\end{equation*}
\item There exists an open set $U\subset M$ (depending on $H$), with $\emptyset
\neq U\cap N \subsetneq N$, and on which $H(U)\equiv 0$. 
\item $0 \leq H(x) \leq m(H)$ for all $x\in M$.
\end{enumerate}

Denote the set of simple Hamiltonians by $\mathcal{H}(M,N)$.
\end{defn}

We now define a \textit{return time} relative to a coisotropic equivalence relation.

\begin{defn}
    \label{def:coisotropic_relation}
Let $(M,\omega)$ be a symplectic manifold, let $N\hookrightarrow M$ be a coisotropic
submanifold and let $\sim$ be a coisotropic equivalence relation on $N$. 
Let $X_H$ denote 
the Hamiltonian vector field of the function $H \colon M\to \R$. 
Suppose $\gamma(t)$ is a solution to $\dot{\gamma}=X_{H}(\gamma)$,
with $\gamma(0) \in N$.

The \emph{return time} of the orbit $\gamma$, relative to $N$ and $\sim$, is defined by
\begin{equation*}\label{E:returnToLeaf}
T_{\gamma} = \inf \{t \,| \,t > 0, \gamma(t) \in N \text{ with } \gamma(0) \sim \gamma(t). \}
\end{equation*}

We define the infimum of the empty set to be $+\infty$.

Notice that if $\sim$ is the trivial equivalence relation, this is a return time
to the submanifold $N$ itself. If $\sim$ is the leaf relation, this measures the
shortest non-trivial leafwise chord.
\end{defn}

We now define admissibility for a simple Hamiltonian,
and use this to define a Hofer-Zehnder-type capacity.
We will find this particularly useful in the case of coisotropic submanifolds
equipped with equivalence relations induced from an ambient coisotropic
submanifold, as in Example \ref{ex:construction_from_ambient}.

\begin{defn} Fix $(M, N, \omega, \sim)$. 
A function $H\in \mathcal{H}(M, N)$ will be called 
\emph{admissible} for the coisotropic equivalence relation $\sim$, 
if all of the solutions of $\dot{\gamma}=X_{H}(\gamma),
\gamma(0)\in N$ are either such that 
\begin{inparaenum}[(i)]
\item $\gamma(t)$ is constant for all $t\in \mathbb{R}$, or 
\item $T_\gamma > 1$, i.e. that the return time of the orbit $\gamma$ relative to $(N,\sim)$ is greater than $1$.
\end{inparaenum} 
We denote the collection of all admissible functions by
$\mathcal{H}_a(M,N,\omega, \sim)$. 
\end{defn}

We now define the map
\begin{defn}
$c(M,N,\omega, \sim)=\sup \{m(H)\,|\,H\in \mathcal{H}_a(M,N,\omega, \sim)\}$.
\end{defn}

Our main theorem is then:
\begin{thm}
The map $c$ is a coisotropic capacity. 
\label{thm:Capacity}
\end{thm}

An application of this theorem together with a computation of capacities 
is the following non-squeezing result for coisotropic 
balls and cylinders which is the natural analogue of the Gromov non-squeezing theorem \cite{Gromov_1985}. To the best of our knowledge, this gives the first relative embedding obstruction with coisotropic constraints which are not Lagrangian.

\begin{figure} 
\includegraphics{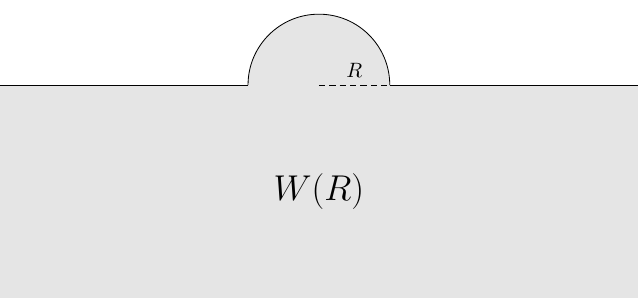}
\caption{The region $W(R)$.}\label{Fig:v_region}
\label{fig:W_R}
\end{figure}

\begin{cor} \label{Cor:non-squeezing}
    Let $B(a,1) \subset \R^{2n}$ be the (open) ball of radius $1$
    centered at $a = (0, \dots, 0, -|a|)$, let $r$ satisfy $|a|^2+r^2 =1$ (so that, in particular, $B^{n,k}(r) = B(a,1) \cap \R^{n,k}$), and suppose that $k < n$.
    
    There exists a relative symplectic embedding 
\[
\phi \colon (B(a,1),B^{n,k}(r),\omega_0) \hookrightarrow (W(R), W^{n,k}(R), \omega_0),
\]
such that any two distinct isotropic leaves of $B^{n,k}(r)$ are mapped to
distinct leaves of $W^{n,k}(R)$
if and only if 
\[
    \arcsin(r) - r(1-r^2)^{1/2} \le \frac{\pi}{2} R^2.
\]

\end{cor}

\begin{rem} \label{rem:area of region bounded by chord}
    The significance of this lower bound becomes clear in the 2-dimensional
    case.
    Consider the disk $B(a,1) \subset \R^2$ 
    of radius $1$ centered at $a$, then $B^{1,0}(r)$
    is (the interior of) a chord of the circle $\partial B(a, 1)$.  This chord
    cuts the disk into two regions.

    The quantity
    \[
    \arcsin(r) - r(1-r^2)^{1/2} 
    \]
    is the area of the smaller of the two regions.

    In this two dimensional case, $W(R)$ is precisely the region shown in Figure
    \ref{fig:W_R}, and $\R^{1,0} = \R$ cuts the region into the lower half-plane
    and an open half-disk of radius $R$. This half-disk has area $\pi R^2/2$.
    This obstruction is therefore obvious in dimension 2. 
    
    Thus, the content of this corollary is that this \textit{a priori} two
    dimensional area obstruction continues to hold in higher dimensions. 
The dynamical origins of the left side of the inequality may be observed in Proposition \ref{prop:Lower bound} and its proof.

    \begin{figure}
        \includegraphics{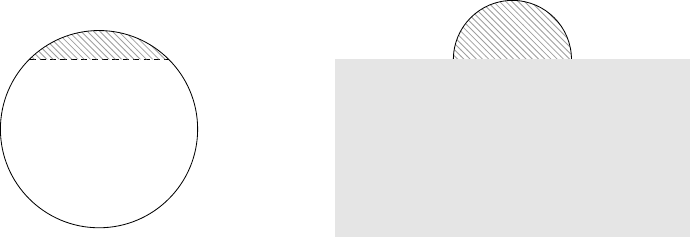}
        \caption{The 2-dimensional case. The hashed area on the left must be
            less than the hashed area on the right for the embedding to exist. 
        } 
    \end{figure}

\end{rem}

Observe that $\R^{2n-2} \times W(1)$ has infinite Gromov width, so this embedding is only obstructed by the coisotropic constraint. 

Several additional applications also follow given the existence of the coisotropic capacity $c$, again using techniques from \cite{Hofer_Zehnder_1994}. 
We give a few of these applications to the existence of chords in Section \ref{sec:Almost existence}. 
In particular, we have:
\begin{restatable*}{thm}{ThmAlmostExistence}

\label{thm:Almost existence}
Let $(M,\omega)$ be a symplectic manifold. Let $S \hookrightarrow M$ be a
compact, regular energy surface for the Hamiltonian $H$. Without loss of
generality, $S = H^{-1}(1)$. 
Let $N \hookrightarrow M$ be an $(n+k)$-dimensional
coisotropic submanifold transverse to $S$, and let $\sim$ be the leafwise
relation on $N$. 

Suppose there is a neighbourhood $U$
of $S$ such that ${c}(U,N,\omega, \sim) < \infty$. 

Then there is a $\rho > 0$ and a dense subset $\Sigma \subset [1-\rho,1+\rho]$ 
such that $X_H$ admits a leafwise chord on every energy surface of $H$ with 
energy in $\Sigma$.

\end{restatable*}

\section{Capacities relative to coisotropic submanifolds}

We now begin the proof of Theorem \ref{thm:Capacity}, giving the monotonicity and conformality axioms, as well as a lower bound. We follow the construction of the Hofer-Zehnder capacity from \cite{Hofer_Zehnder_1994}. 
We also provide a proof of Corollary \ref{Cor:non-squeezing}.

Let $(M,N)$ be a pair consisting of a symplectic
manifold $(M,\omega)$ and a properly embedded coisotropic submanifold $N\hookrightarrow M$, i.e. with $\partial
N \subset \partial M$ (or $\partial N = \emptyset$).  All of
our symplectic manifolds will be assumed to be of the same dimension $2n$. 

We begin with a few definitions.

\begin{defn}
\label{defn:Coisotropic objects}
Recall that \[\R^{n,k} \ceq \{x \in \R^{2n}| x = (x_1,\dots,x_n,y_{1}, \dots, y_{k}, 0,\dots,0) \}\]
is an $(n+k)$-dimensional coisotropic linear subspace of $\R^{2n}$, and that $B(a,r)$ is the $2n$ dimensional symplectic ball of radius $r$ centered at 
\[a \ceq (0,\dots,0,b_n) \in \R^{2n},\] 
and $B^{n,k}(r) \subset \R^{n,k}$ as the coisotropic ball of radius $r$ centered at $0 \in \R^{n,k}$:
\[
B^{n,k}(r) \ceq \{ x\in \R^{n,k}\, \vert \, |x| \,\leq r \}. 
\]
\begin{figure}
\includegraphics{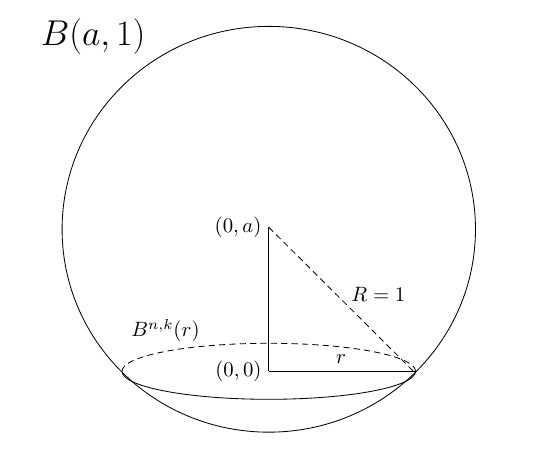}
\caption{A schematic image of the embedding of $B^{n,k}(r)$ into $B(a,1)$.} 
\label{fig:standard balls}
\end{figure}
Recall that we likewise denote by
\[Z(a,r) \ceq \{ x\in \R^{2n} \, \vert \, x_n^2 + (y_n - b_n)^2 \leq r^2\}
\]
the symplectic cylinder centered at $a\in \R^{2n}$, and by $Z^{n,k}(r)$ the coisotropic cylinder 
\[ 
Z^{n,k}(r) \ceq \{ x \in \R^{n,k} \, \vert \, |x_n| \leq r\}
\]
\end{defn}
\begin{rem}
Note that $B^{n,k}(r)$ and $Z^{n,k}(r)$ are properly embedded in $B(a,1)$, $Z(a,1)$, respectively, when $a = (0,\dots,0,b_n)$ with $|b_n| < 1$, and $r^2 = 1 - |a|^2$.
See Figure \ref{fig:standard balls}.
\end{rem}

We will now study the properties of the map $c(M, N, \omega, \sim)$. 
We will show that this satisfies the axioms for coisotropic capacities.

We prove the monotonicity and conformality properties below, which proceed as in \cite{Hofer_Zehnder_1994}. The proof of 
non-triviality will be completed in Section \ref{sec:Capacity proof}.

\begin{lem}The map $c$ satisfies the monotonicity axiom.
\label{lem:monotonicity}
\end{lem}
\begin{proof}
Let $\phi \colon (M_1,N_1,\omega_1, \sim_1)\to (M_2,N_2,\omega_2, \sim_2)$ 
be a relative embedding, respecting the coisotropic equivalence relations, as in
Definition \ref{def:relative_embedding}. 
Define the map $\phi_* \colon \mathcal{H}(M_1,N_1)\to
\mathcal{H}(M_2,N_2)$
by
\begin{equation*}
\phi_*(H) = \begin{cases} H \circ \phi^{-1} & x\in \phi(M_1) \\
m(H) & x\notin \phi(M_1)
\end{cases}
\end{equation*}
Note that if $A\subset M_1 \setminus \partial M_1$ is a non-empty compact set 
and $\emptyset \subsetneq A\cap N_1 \subsetneq N_1$, then
$\phi(A)\subset M_2 \backslash \partial M_2$ and $\emptyset \subsetneq \phi(A)
\cap N_2 \subsetneq N_2$. By construction, $m(H)=m(\phi_*(H))$. If $U\subset
M_1$ is an open set on which $H=0$, then, since $\phi$ is an embedding $\phi(U)$
is an open set on which $\phi_*(H)=0$. Also, by construction, $0\leq \phi_*(H)
\leq m(\phi_*(H))$, and therefore $\phi_*(H)\in \mathcal{H}(M_2,N_2,\omega_2)$. 

We now check that $\phi_*(\mathcal{H}_a(M_1,N_1,\omega_1, \sim_1))\subset
\mathcal{H}_a(M_2,N_2,\omega_2, \sim_2)$. 
Let $H \colon M_1 \to \R$ be an admissible simple Hamiltonian.
Since $\phi$ is symplectic,
$\phi_*(X_H)=X_{\phi_*(H)}$. 
Furthermore, the Hamiltonian vector field $X_{\phi_*(H)}$ vanishes outside the
image of $\phi$. 
Thus, all non-constant trajectories of $X_{\phi_*(H)}$ are conjugate to
non-constant trajectories of $X_H$. 
In particular then, 
if $y(t)$ is a non-constant trajectory of $X_{\phi_*(H)}$ with 
$y(0) \in N_2, y(T) \in N_2$ with $y(0) \sim_2 y(T)$, then $y(t)$ must be in the
image of $\phi$ and thus there exists a
trajectory $x(t)$ of $X_H$ so that $\phi(x(t)) = y(t)$.

Since $\phi$ is a relative embedding with $\phi^{-1}(N_2) = N_1$, we have that
$x(0), x(T) \in N_1$. 
Since the relative embedding $\phi$ respects the coisotropic equivalence
relations, if $y(0) \sim_2 y(T)$ then it must hold that $x(0) \sim_1 x(T)$.
As $H \in \mathcal{H}_a(M_1, N_1, \omega_1, \sim_1)$, it follows that $T > 1$.
Hence, it follows that $\phi_*H \in \mathcal{H}_a(M_2, N_2, \omega_2, \sim_2)$,
as desired.

\end{proof}

We now give a slight extension of the above monotonicity,
which will be of use to us in the proof of Theorem \ref{thm:Upper bound}.
\begin{lem} \label{lem:referee request}
    Let $(M, \omega)$ and $(M', \omega')$ be symplectic manifolds, let $N
    \subset M$ and $N' \subset M'$ be coisotropic submanifolds equipped with
    coisotropic equivalence relations $\sim$ and  $\sim'$.

    Suppose that for every compact set $K \subset M$, there exists
    an open neighbourhood $U \supset K$ and a relative 
    symplectic embedding $\psi \colon (U, N \cap U, \omega|_U)  \to (M, N, \omega)$ that respects the pair of coisotropic equivalence relations $(\sim, \sim')$.

    Then, $c(M, N, \omega, \sim) \le c(M', N', \omega', \sim')$.
\end{lem}
\begin{proof}
    Let $H \colon M \to \R$ be a Hamiltonian with $0 \le H \le m(H)$ and that $m(H) - H$ 
    is compactly supported in $M$. Then, by hypothesis, there exists a
    neighbourhood $U$ of the support of $m(H) - H$ and a symplectic embedding
    $\psi \colon U \to M'$. Let $G = H \circ \psi^{-1}$ defined on $\psi(U)$ and
    then extend $G$ to all of $M'$ by setting $G(x) = m(H)$ for all $x \notin
    \psi(U)$. 

    Proceeding as in Lemma \ref{lem:monotonicity}, it follows that $G$ is simple
    if and only if $H$ is simple, with $m(G) = m(H)$. Furthermore, $X_{G} =
    \psi_*X_H$ on $\psi(U)$ and vanishes outside $\psi(U)$. Thus, arguing as in
    Lemma \ref{lem:monotonicity}, $G$ is admissible if and only if $H$ is.

    Thus for any $H \in \mathcal{H}_a(M,N,\omega, \sim)$, there exists $G \in 
\mathcal{H}_a(M',N',\omega', \sim')$ such that 
    $m(H) = m(G)$. The desired inequality now follows immediately.
\end{proof}

\begin{rem} \label{rem:leaf_relation_monotonicity}
The monotonicity of the capacity depends in an essential way on the 
coisotropic equivalence relation. 
Indeed, if no condition is put on the relative embedding 
$\phi \colon (M_1, N_1, \omega_1) \to (M_2, N_2, \omega_2)$,
it is easy to imagine a situation in which two or more leaves on $N_1$ are
mapped to the same leaf in $N_2$. 
For instance, there are many examples of compact hypersurfaces in $\R^{2n}$
for which there is a dense leaf in the characteristic foliation --- 
in this case, this is about dense orbits in a Hamiltonian system with compact
energy level. 
One possible construction is originally due to Katok \cite{Katok_1973},
as is used in \cite{Casals_Spacil_2016}.
In particular, Katok's construction can be done as a small, locally supported
perturbation of the unit sphere in $\R^{2n}$. 
This construction of Katok's also shows that the existence of leafwise chords
must see the entirety of the coisotropic submanifold.

As discussed in Example \ref{ex:construction_from_ambient}, 
a natural class to consider is to consider a fixed (compact) coisotropic
submanifold $\hat N$ in $\R^{2n}$. We then consider only symplectic manifolds
that are open subsets $U \subset \R^{2n}$ and coisotropic submanifolds $N = \hat
N \cap U$. The coisotropic equivalence relation is that $x \sim y$ if and only
if $x, y$ are in the same isotropic leaf on $\hat N$.
Then, all of the inclusion maps respect the equivalence relation, by
construction.

A very simple example of this phenomenon occurs even with Lagrangians.
Let $(\hat M, \hat N)$ be the pair consisting 
of the unit disk in $\R^2$ and the $x$-axis.
Let $M$ be an open annulus centred at the origin. Then, $N = \hat
N \cap M$ is the disjoint union of two line segments. 

In $M$, each line segment is its own leaf: an admissible Hamiltonian for the
leafwise relation, just considered relative to $N$, would allow for a finger move that
pushed centre of the segments almost all the way around the annulus.

Notice however that the inclusion of $(M, N) \hookrightarrow (\hat M, \hat N)$
does not respect the leafwise relation! The two leaves of $\hat N$ are both
mapped to the unique leaf of $N$. 
Relative to the equivalence relation induced from the leafwise relation on $\hat N$, however, 
these chords from one segment to the other would be eliminated.  

We thank Kaoru Ono and Yoshihiro Sugimoto for pointing out that 
our original version of this capacity $c$ overlooked this point and 
implicitly required the embeddings to
respect the leaf relation.
\end{rem}

\begin{lem} \label{lem:conformality}
$c$ satisfies the conformality axiom.
\end{lem}
\begin{proof}
    Let $\alpha \ne 0$.
    Define a map $\psi \colon \mathcal{H}(M,N) \to \mathcal{H}(M,N)$ by 
\[ 
\psi(H) \ceq |\alpha|\cdot H,
\] 
and let $H_{\alpha}$ denote $\psi(H)$. 

Note that $\psi$ is clearly injective, and
$m(H_\alpha )=| \alpha | m(H)$, so the lemma follows if
\[
    \psi|_{\mathcal H_{a}(M, N, \omega, \sim)} 
    \colon \mathcal H_{a}(M, N, \omega, \sim) \to  
        \mathcal{H}_{a}(M, N, \alpha \omega, \sim)
\]
is a bijection.
Let $X_{H_\alpha}$ be the Hamiltonian vector field generated by $H_\alpha$ 
with respect to $\alpha \omega$, in other words such that $\alpha \omega( X_{H_{\alpha}}, \cdot) =
-dH_\alpha$.
Hence, 
\begin{align*}
    &\alpha \omega( X_{H_\alpha}, \cdot) = - |\alpha| dH 
    &\omega(X_{H_\alpha}, \cdot) = -\frac{|\alpha|}{\alpha}dH.
\end{align*}
Thus, $X_{H_\alpha} = \pm X_H$, depending on the sign of $\alpha$.
Therefore, the Hamiltonian flows for $H$ and $H_{\alpha}$ have the same orbits.
In particular, the constant chords are the same for the two Hamiltonians.
A non-constant chord for one of them, $x(0) \in N$, $x(T) \in N$ with $x(0) \sim
x(T)$, will be a chord for the other, 
by considering $x(t)$ itself if $\alpha > 0$ and the time reversal 
$t \mapsto x(T-t)$ if $\alpha < 0$.
Their return times are thus the same.
\end{proof}

In the next proposition, we give a lower bound for
$c\left(B(a,1),B^{n,k}(r),\omega_0\right)$ 
with $r = \sqrt{1-|a|^2}$.

\begin{prop}
\label{prop:Lower bound}
Let $a \ceq (0,\dots,0,b) \in \R^{2n}$, and set $r^2 = 1 - |a|^2 = 1 - |b|^2$.
For $k\in \{0,\dots,n-1\}$,
\[c\left(B(a,1),B^{n,k}(r),\omega_0 \right) \geq  \arcsin(r) - r(1-r^2)^{1/2}.
\]
\end{prop}

\begin{proof}

We consider first when $|a| > 0$. We suppose, without loss of generality, that $b < 0$ and thus $b = -|a|$.

We will construct a family of Hamiltonian functions all of which are admissible and whose maximum is arbitrarily close to $\arcsin(r) - r(1-r^2)^{1/2}$. First, decompose $\R^{2n} = \R^n \oplus J\R^n = (x_1,\dots, x_n,y_1,\dots,y_n)$, where $J:\R^{2n} \to \R^{2n}$
\begin{align}
\begin{split}
\label{eq:J in Rn}
J(0,\dots,0,x_i,0,\dots,0) &= (0,\dots,0,y_i,0,\dots,0)\\
J(0,\dots,0,y_i,0,\dots,0) &= (0,\dots,0,-x_i,0,\dots,0),
\end{split}
\end{align}
and we understand $J\R^n$ to indicate $J$ applied to the first $n$ dimensions of $\R^{2n}$.
Choose $\epsilon > 0$,
and let $f\colon [0,1] 
\subset \R \to \mathbb{R}$
be a function with the following properties:
\begin{alignat*}{3}
& f(t)=0 &&\text{ for } t \in [0,|a|+\epsilon], \\
& 0 \leq f'(t) < \arccos \left(\frac{|a|}{\sqrt{t}} \right) \quad && \text{ for } t > |a| + \epsilon,\\
&f(t) = \max f \quad &&\text{ for } t \in [1-\epsilon, 1].
\end{alignat*}

Let $H \colon \R^{2n}\to \R$ be the Hamiltonian defined by $H(x) \ceq f(|x-a|^2)$.

We will first observe that $H$ is simple. Note first of all that $H = 0$ in an open ball around $a$, and this ball intersects $B^{n,k}(r)$, as required. Also observe that $H = \max f$ once $|x-a|^2 \ge 1-\epsilon$,
so this gives $H = \max f$ in a collar neighbourhood of the boundary of $B(a,1)$ as required.

We will now show that any such Hamiltonian $H$ will be admissible. 

We consider the associated Hamiltonian ODE given by
\begin{align*}
\dot{x} = J\nabla H(x) =& 2 f' \left(|x-a|^2\right)J(x-a)
\end{align*}
where $J \colon \R^{2n} \to \R^{2n}$ is the standard almost complex structure defined by Equation \ref{eq:J in Rn} above.
Since $\langle Jx, x \rangle = 0$, we see that $|x-a|^2$ is constant along solutions of the equation. Thus, with $z(t) \ceq x(t) - a$
we have for $\kappa = 2f'( |z(0)|^2 ) \ge 0$, 
\begin{equation} 
\dot z =  \kappa J z.
\end{equation}
Thus, $z(t) = \e^{\kappa Jt} z(0)$.

If a trajectory $z(t)$ starts on the coisotropic submanifold, we then have the initial conditions
\[
z(0) \in \R^{n,k}.
\]

To verify admissibility, we will show that every non-constant trajectory starting on the coisotropic submanifold has (non-leafwise, coisotropic) return time greater than $1$.

Let $z(t)$ be such a non-constant trajectory, with $z(0)\in B^{n,k}(r)$. It follows then that $\kappa \ne 0$.  Now consider the triangle formed by the origin, $a$, and $c= (0,\dots,-r,0,\dots,0)$, where the $-r$ is in the $n$-th position. Note that, if we consider the plane defined by these three points, then any flow $z(t)$ with $z(0)$ on the line from $a$ to $c$ flows counterclockwise in this plane. Since $f$ is a radial function, we may, without loss of generality, simply consider any such flow $z(t)$ with $z(0)$ on this line.

\begin{figure}
\tdplotsetmaincoords{11}{0}
\begin{tikzpicture}[scale=3.5]
	\path [use as bounding box] (-1.2,-0.4) rectangle (1.3,1.83);
	\draw (0,0) -- (0,0.71) node [anchor=north] at (0,0) {$0$}
		node [anchor= south] at (0,0.71) {$a$};
	\draw (0.71,0) -- (-0.71,0) node [anchor=south] at (0.5,0) {$B^{n,k}(r)$};
	\draw (0,0.7) circle[radius=1];
	\node at (-0.1,0.64) {$\theta$};
	\filldraw (-0.4,0.71) circle (0.5pt) node [anchor=south] {$x(0)$}; 
	\draw (-1,0.71) -- (1,0.71);
	\draw (-0.4,0) -- (-0.4,0.71);
	\draw [densely dashed] (0,0.71) -- (-0.4,0) node [anchor=south] at (-0.26,.3) {$\rho$}; 
	\filldraw[black] (-0.4,0) circle (0.5pt) node [anchor=north] {$z(0)$};
	\node at (-0.85,1.65) [node font = \LARGE]{$B(a,1)$};
\end{tikzpicture}
\caption{}
\label{fig:Lower bound}
\end{figure}

 Let $\rho = \sqrt{|x(0)|^2+|a|^2} = |z(0)| \geq |a|$, and let $\theta$ be the angle so that $|x(0)| = \rho \cos(\theta)$ and $|a| = \rho \sin(\theta)$. See Figure \ref{fig:Lower bound} for an illustration. Hence, $\theta = \arcsin(\frac{|a|}{\rho})$. 
 If $T$ is such that $z(T)$ belongs to $\R^{n,k}$, we have $\sin(\kappa T + \theta) = \sin( \theta) = \frac{|a|}{\rho_z}$, which holds if and only if $\kappa T \in 2 \pi \Z$ or $\kappa T = -2\theta + k\pi$ for some $k$ odd. In particular then, if $\kappa < \pi - 2 \theta$, there can be no chords of length at most $1$. Recall now that $\kappa = 2f'( |z(0)|^2 )$. Thus, the condition is satisfied if we have $2 f'( \rho^2 ) < \pi - 2\theta$ for each $\rho_z$. This is achieved if 
\[
f'( \rho^2) < \frac{\pi}{2} - \theta =  \arccos \left(\frac{|a|}{\rho}\right)
\]
However, by assumption, $f'(\rho_z^2) < \arccos\left(\frac{|a|}{\rho_z}\right)$.
Observe now that by choosing $\epsilon > 0$ sufficiently small, we may arrange
for $f(1)$ to be arbitrarily close to 
\[
\begin{aligned}
\int_{|a|^2}^1 \arccos\left( \frac{|a|}{\sqrt{t}}\right) dt 
&= \int_0^{\arccos(|a|)} 2 |a|^2 \alpha \cos(\alpha)^{-3} \sin(\alpha) d \alpha \\
&= |a|^2 \left (  \alpha \cos(\alpha)^{-2} - \tan \alpha  \right )  \Big \vert_0^{\arccos(|a|)}  \\
&= \arcsin r - r \sqrt{1-r^2}.
\end{aligned}
\]
(Recalling that $a^2 + r^2 =1$.)

From this, we conclude
\[
c(B(a,1),B^{n,k}(r),\omega_0) 
\geq \arcsin (r) -  r(1-r^2)^{1/2},
\]
as desired, in the case that $|a| > 0$.

If $a = 0$, we observe that for each $\delta > 0$, we may set $p = (0, \dots, 0, -\delta)$
and then we consider the inclusion of the ball $B(p, 1-\delta) \subset B(0, 1)$. 
The intersection of $B^{n,k}(r)$ with this smaller ball is given by $B^{n,k}( \sqrt{1 - 2\delta}) + p$.
After a translation of the origin, this gives a relative embedding of the pair $(B(-p, 1-\delta), B^{n,k}(\sqrt{1-2\delta}) )$. Let $r_\delta = \sqrt{1 - 2\delta}$. Then, applying the above construction and the Monotonicity Axiom (Lemma \ref{lem:monotonicity}), we have for each $\delta > 0$, 
\[
c(B(0,1),B^{n,k}(1),\omega_0) \geq \arcsin (r_\delta) -  r_\delta (1-r_\delta^2)^{1/2}
\]
Taking $\delta \to 0$, we obtain $c(B(0,1), B^{n,k}(1) ) \ge \arcsin(1) = \frac{\pi}{2}$, proving the result.

\end{proof}

\begin{proof}[Proof of Corollary \ref{Cor:non-squeezing}]

    By the non-triviality and conformality axioms for the capacity,
    we obtain that $c(W(R), W^{n,k}(R)) = \frac{\pi}{2} R^2$. 

    The monotonicity of the capacity $c$ and Proposition \ref{prop:Lower bound}
    then give that a relative embedding $(B(a,1), B^{n,k}(r)) \hookrightarrow
    (W(R), W^{n,k}(R))$ respecting the leaf relations  exists only if 
    \[
        \arcsin (r) -  r(1-r^2)^{1/2} \le \frac{\pi}{2}R^2.
    \]

    To prove that this suffices, we will construct an embedding for any $R$ that
    satisfies
    \[
        \arcsin (r) -  r(1-r^2)^{1/2} < \frac{\pi}{2}R^2.
    \]

    By a slight abuse of notation (since $a \in \R^{2n}$), 
    let $D(a, \rho) \subset \R^2$ be the disk of radius $\rho$ centred at $(0,
    -|a|)$. 

    First, we notice that the ball embeds in an appropriate polydisk:
    \begin{align*}
        B(a,1)\subset& D^2(0,1) \times \dots D^2(0, 1) \times D^2(a, 1)\\
        \begin{split} =&\{ (x_1, \dots, x_n, y_1, \dots, y_n) \, | \\ & \quad x_1^2+y_1^2 < 1,
            \dots, x_{n-1}^2 + y_{n-1}^2 < 1, (x_n+a)^2 +
    y_n^2 < 1 \}. \end{split}
    \end{align*}
    This respects the leaf relation on $\R^{n,k}$. 

    We will now construct an embedding 
    \[
        \psi \colon D^2(0,1) \times \dots D^2(0, 1) \times D^2(a, 1) \to W(R)
    \]
    of the form 
    \[
        \psi(x_1, \dots, x_{n-1}, x_n, y_1, \dots, y_{n-1}, y_n) = (x_1, \dots,
        x_{n-1}, f(x_n, y_n), y_1, \dots, y_{n-1}, g(x_n, y_n))
    \]
    for a suitable choice of map
    \[
        \phi \colon \R^2 \to \R^2, \phi(x,y) = (f(x,y), g(x,y)).
    \]
    Let $W^2(R) \coloneqq \{ (x, y) \in \R^2 \, | \, x^2+y^2 < R^2 \text{ or }
    y < 0 \}$.  
    Observe now that $\psi$ is symplectic if and only if $\phi$ is area
    preserving. 
    Furthermore, $\psi$ gives a relative embedding of the polydisk into $W(R)$
    (with coisotropic submanifold given by the restriction of $\R^{n,k}$ to
    each)
    if and only
    $\phi \colon (D(a, 1), \R \cap D(a,1)) \to (W^2(R), \R \cap W^2(R))$ is a
    relative embedding. Finally, we observe that if $\phi$ is such a relative
    embedding, it immediately follows from the explicit description of the leaf
    relation in Example \ref{ex:Isotropic leaves of Rnk} that $\psi$ respects the leaf relation.

    It suffices therefore to find an embedding $\phi \colon D(a,1) \to W^2(R)$.
    By a standard Moser-type argument, this exists whenever the area of the
    smaller of the two connected components of $D^2(a,1) \setminus \R$ is
    strictly smaller than the area of the upper half disk in $W^2(R) \setminus
    \R$. The result now follows by computing this area, as in Remark
    \ref{rem:area of region bounded by chord}.

\end{proof}

\section{An upper bound for $c\left( U(r), U^{n,k}(r), \omega_0, \sim\right)$}
\label{sec:Capacity proof}

In the following, we will write $c( M, N) = c(M, N, \omega,\sim)$
since we are considering subsets $M \subset \R^{2n}$ with respect to the
standard symplectic form. Furthermore, we will take the equivalence relation to
be the leafwise equivalence relation.

In order to show that $c$ is a coisotropic capacity, 
we must establish the non-triviality axiom. 
Recall that we have defined 
\[
    W(1) = \left \{ (x_1, \dots, x_n, y_1, \dots, y_n) \in \R^{2n} \; | \; x_n^2 + y_n^2  < 1 \text{ or } y_n < 0 
\right\}
\]
and $W^{n,k}(1) = W(1) \cap \R^{n,k}$, with $\R^{n,k}$ 
the standard $(n+k)$-dimensional coisotropic subspace of $\R^{2n}$, given by 
\[
\R^{n,k} = \left \{ (x_1, \dots, x_n, y_1, \dots, y_k, 0, \dots, 0) \right \}.
\]

By the relative symplectic embedding of the ball 
\[
(B(1), B^{n,k}(1)) \hookrightarrow (W(1), W^{n,k}(1) ), 
\]
and monotonicity  (Lemma \ref{lem:monotonicity}), together with Proposition \ref{prop:Lower bound}, it suffices to prove the following inequality:
\[
c(W(1), W^{n,k}(1) )  \le \frac{\pi}{2}.
\]

For our analytical set-up, it is most convenient to work with the region $U(1)$ in $\R^{2n}$ given as the union of the disk with a half-infinite strip
\[
U(1) = \R^{2n-2} \times \{ (x,y) \in \R^2 \, | \, x^2 + y^2 < 1 \text{ or } -1 < x < 1 \text{ and } y < 0 \}
\]
and $U^{n,k}(1) = U(1) \cap \R^{n,k}$. 
In the following, we will write $U = U(1)$ and $U^{n,k} = U^{n,k}(1)$.

We claim that the relative capacities of these two domains are the same:
\[
c(W(1), W^{n,k}(1) ) = c( (U(1), U^{n,k}(1) ) ).
\]
Observe that there is a relative embedding 
\[
(U(1), U^{n,k}(1)) \hookrightarrow (W(1), W^{n,k}(1) )
\]
by inclusion, showing one inequality. 
The other inequality is by applying Lemma \ref{lem:referee request}. 
Indeed, for any compact set $K \subset W(1)$, by a Moser argument, 
we may find an open neighbourhood $V$ and a symplectic embedding 
$V \hookrightarrow U(1)$
that may be taken to the the identity in the region $y_n > -\delta$ for $\delta
> 0$ sufficiently small, and hence is the identity on the coisotropic
submanifold. The existence of such a symplectic embedding for each compact $K
\subset W(1)$ then verifies the hypotheses of the Lemma, and the claim
follows. 

\begin{prop} \label{prop:Upper Bound}
The map $c$ verifies
\[
c( U, U^{n,k} ) \le \frac{\pi}{2}.
\]
\end{prop}

As explained above, this will then prove Theorem \ref{thm:Capacity}.
The remainder of this section will prove Proposition \ref{prop:Upper Bound}.

\subsection{The analytical setting}

\begin{defn}
We recall from Example \ref{ex:Isotropic leaves of Rnk} that $\omega_0$ denotes the standard symplectic form on $\R^{2n}$, and
that $\R^{n,k}$ is the linear coisotropic subspace of $(\R^{2n},\omega_0)$ consisting
of the first $n+k$ coordinates, i.e. 
\begin{equation*}
\R^{n,k} = \left \lbrace x \in \R^{2n} \; \vert  \; 
x = (x_1,\dots,x_n,y_1,\dots, y_k,0,\dots,0) \right \rbrace.
\end{equation*}
Let $V_0, V_1$ and $W_0$ be the linear subspaces
\begin{align*}
V_0 &= \left \lbrace x \in \R^{2n} \; 
\vert \; x = (0,\dots,0,x_{k+1},\dots,x_n,0,\dots,0) \right \rbrace, \\
V_1 &= \left \lbrace x \in \R^{2n} \; 
\vert \; x=(x_1,\dots,x_k,0,\dots,0,y_1,\dots,y_k, 0\cdots,0) \right \rbrace \\
W_0 &= \left \lbrace \in \R^{2n} \; 
\vert \; x=(0,\dots,0,y_{k+1},\dots,y_n) \right \rbrace.
\end{align*}
\begin{rem}As noted in Example \ref{ex:Isotropic leaves of Rnk}, any leaf $F$ in the characteristic foliation has the form $z + V_0$, 
for $z \in \R^{n,k}$.
\end{rem}
Let $C_{n,k}^{\infty}\left([0,1]\right)$ denote the
space of smooth maps $\psi:[0,1]\to \R^{2n}$ such that $\psi(0),\psi(1) \in
F \subset \F$ for some isotropic leaf $F$ in the characteristic
foliation $\F$ of $\R^{n,k}$. Let $\langle\cdot, \cdot \rangle$ be the standard
inner product on $\R^{2n}$, and define the functional 
$\Phi_H \colon C_{n,k}^{\infty}\left([0,1]\right) \to \R$ by
\begin{equation}
\label{eq:Phi_H}
\Phi_H(\psi) = \frac{1}{2}\int_0^1 \langle -J\dot{\psi}(t),\psi(t) \rangle dt -
\int_0^1 H(\psi(t)) dt.
\end{equation}
\end{defn}

In order to study the critical points of $\Phi_H$, we will extend the
definition of $\Phi_H$ to the Hilbert space of $H^{1/2}$ paths. 
The Hilbert space is constructed so the paths have boundary in $\R^{n,k}$, even though 
$H^{1/2}$ does not embed in $C^0$, and thus a pointwise constraint cannot be imposed.
The key observation we use is that $\R^{n,k}$ is the fixed point locus of an involution on $\R^{2n}$,
which then induces an isometry on $H^{1/2}(S^1, \R^{2n})$. Our path space is then an eigenspace 
of this isometry, though we also describe it explicitly.

We first show the
following.
\begin{lem}
\label{lem:Fourier decomposition}
Any element $\gamma \in
C^{\infty}_{n,k}\left([0,1]\right)$ is given by 
\begin{equation}
\label{eq:Fourier decomposition}
\gamma(t) = \sum_{k\in \Z} e^{k\pi J t}a_k  + \sum_{k\in 2\Z} e^{k\pi J t}b_k 
\end{equation}
where
\begin{equation}
\label{eq:Boundary conditions}
\begin{split}
&a_k \in V_0 \subset \R^{n,k} \subset \R^{2n}, \text{ and }\\
&b_k \in V_1 \subset \R^{n,k} \subset \R^{2n}.
\end{split}
\end{equation}
Equivalently, 
\begin{equation*}
\gamma(t) = \sum_{k\in \Z} z_k e^{k\pi J t}  
\end{equation*}
with $z_k \in V_0$ for odd $k$ and $z_k \in V_0 \oplus V_1$ 
for even $k$ (i.e.~$z_k = a_k + b_k$ with $b_k = 0$ for all odd $k$).
\end{lem}
\begin{proof}
We begin by identifying $\R^{2}$ with $\C$, and we consider a smooth map
$\gamma(t):[0,1]
\to \C$ such that $\gamma(0), \gamma(1) \in \R \subset \C$. We now extend this
map to a piecewise smooth map $\alpha(t):S^1 \to \C$ by
\begin{equation*}
 \alpha(t) = \begin{cases} \gamma(2t) & t \in \left[0,\frac{1}{2}\right] \\
 \overline{\gamma(2-2t)} & t \in \left(\frac{1}{2},1\right],
             \end{cases}
\end{equation*}
where the bar indicates complex conjugation. Note that $\alpha(t)$ is
continuous by definition. Writing $\alpha(t)$ in terms of its Fourier
decomposition, we have
\begin{equation*}
\alpha(t) = \sum_k e^{2\pi i k t} a_k.
\end{equation*}
However, since $\alpha(t) = \overline{\alpha(1-t)}$, and therefore
\begin{align*}
 \sum_k e^{2\pi i k t}a_k &= \sum_k e^{-2\pi i k (1-t)}\overline{a_k} \\
 &= \sum_k e^{-2\pi i k }e^{2\pi i k t}\overline{a_k} \\
 &= \sum_k e^{2\pi i k t}\overline{a_k},
\end{align*}
which implies that $a_k = \overline{a_k}$, and therefore $a_k \in \R \subset
\C$. Our original function $\gamma(t)$ is recovered by $\gamma(t) = \alpha(t/2)
= \sum_k e^{\pi i k t}a_k$, where $a_k \in \R$.

Now consider a function $\gamma(t):[0,1] \to \R^{2n}$ such that $\gamma(0),
\gamma(1) \in F$,
where $F$ is a leaf of the characteristic foliation of $\R^{n,k}$. 
Write a point $x\in \R^{2n}$ by
$(x_1,\dots,x_n,y_1,\dots,y_n)$, where
$\omega_0(\frac{\partial}{\partial x_i},\frac{\partial}{\partial y_i}) = 1$, $J
\frac{\partial}{\partial x_i}=\frac{\partial}{\partial y_i}$, for $J$ the standard complex
structure on $\R^{2n}$, and define $c_{n,k}:\R^{2n} \to \R^{2n}$ by
\begin{equation*}
 c_{n,k}(x) := (x_1,\dots,x_n,y_{1},\dots,y_k,-y_{k+1},-y_n).
\end{equation*} 
Recall that $\R^{n,k}$
is the set of points 
\begin{equation*}
\R^{n,k} = \{ x \in \R^{2n} | x = (x_1,\dots,x_n,y_1,\dots,y_{k},0,\dots,0)\}.
\end{equation*}
In the special case of a Lagrangian, i.e.~for $\R^{n,0}$, we note that $c_{n,0}$ 
is a real structure for $\omega_0$, i.e. $c_{n,0}^*\omega_0 = -\omega_0$. 

Any leaf $F$ of $\mathcal{F}$ is a set of the form
\begin{equation*}
\{x \in \R^{n,k} \,|\, x = (0,\dots,0,x_{k+1},\dots,x_n,0,\dots,0) + z\}
\end{equation*}
for some fixed $z = (x_1,\dots,x_k,0,\dots,0,y_1,\dots,y_k,0,\dots,0)$.
We may write
$\gamma(t)$ as a function $\gamma(t) = z_1(t) + z_2(t) + \dots + z_n(t)$, where each
function
$z_i \colon  [0,1] \to \R^{2n}$ is a map $t \mapsto
(0,\dots,0,x_i(t),0,\dots,0,y_i(t),0,\dots,0)$ for real functions $x_i,y_i:[0,1]
\to\R$.

From the above, we see that if $i > k$, then 
\begin{equation*}
z_i(t) = \sum_j e^{J\pi jt}a_{i,j}
\end{equation*}
where $a_{i,j} = a_{j}e_{i}$ for constants $a_{j}\in \R$, $e_i$ a vector with $1$
in the $i$-th position and $0$s elsewhere. This then gives that $a_{i,j} \in V_0$.

For $i\leq k$, $z_i(0)=z_i(1)$, and we have
\begin{equation*}
z_i(t) = \sum_j e^{2\pi j J t}a_{i,j},
\end{equation*}
where $a_{i,j} = a_{j}e_{i}$ with $a_j \in \C$.
From this, we have that $a_{i,j} \in V_1$.

The conclusion of the lemma now follows immediately.
\end{proof}

\begin{rem}
Note that if $\gamma \in C^0([0,1],\R^{2n}) \cap L^{1}$ and is of the form
\[
\gamma(t) = \sum_{k\in \Z} e^{k\pi J t}a_k + \sum_{k\in 2\Z} e^{k\pi J t}b_k
\] with
$a_k, b_k$ as in Equation \ref{eq:Boundary conditions} above, then necessarily 
$\gamma(0), \gamma(1) \in F$.
\end{rem}
\begin{defn} Let $L^2_{n,k}([0,1])$ be the Hilbert space
\begin{equation*}
\begin{split}
L^2_{n,k} = \bigg\{\gamma \in L^2([0,1],\R^{2n}) \, \bigg\vert \, & 
    \gamma = \sum_{k\in \Z} a_k e^{k\pi J t} + \sum_{k\in 2\Z} b_{k} e^{k\pi J t},\\
& \, a_k\in V_0, \,b_{k} \in V_1,\\
& \sum_{k\in \Z} \vert a_k \vert^2 + \vert b_k \vert^2 < \infty \bigg \}
\end{split}
\end{equation*}
with inner product 
\begin{equation*}
    \left \langle \psi,\phi \right \rangle_{L^2_{n,k}} =  \left (
    \int_0^1 \langle \psi(t),\phi(t) \rangle \, dt \right)^{\frac{1}{2}}.
\end{equation*}
Define $H_{n,k}^{s}\left([0,1]\right)$ to be the space
\begin{equation*}
H_{n,k}^{s}([0,1]) = \left\lbrace x\in L^2_{n,k} \left([0,1]\right)\: \bigg\vert
\;
\sum_{k\in \Z} |k|^{2s}|z_k|^2 < \infty \right\rbrace
\end{equation*}
where $z_{k} \in V_0$ for odd $k$ and $z_k \in V_0 \oplus V_1$ for even $k$.
\end{defn}

In the following lemmas, we collect several standard results from
\cite{Hofer_Zehnder_1994} concerning 
the spaces 
$H^s_{n,k}([0,1])$. The proofs are identical to those in \cite{Hofer_Zehnder_1994}, 
replacing the spaces considered there with
the corresponding spaces in our setting. For the convenience of the reader, we have tried to keep our notation
compatible with the notation of \cite{Hofer_Zehnder_1994}*{Sections 3.3, 3.4}.
One notable change is that we use $X$ to denote the appropriate
$H^{\frac{1}{2}}$ Hilbert space, which is denoted by $E$ in
\cite{Hofer_Zehnder_1994}. Some of the more immediate results are stated without proof. 
\begin{defn} \label{def:X Hilbert space with norm}
    Denote by \[
    X = H^{1/2}_{n,k}\left ( [0,1] \right ). \]

    For $\gamma \in X$, we have 
    \[
    \gamma = \sum_{k\in \Z} z_k e^{k\pi J t} 
   \]
   where
   $z_k \in V_0$ for odd $k$ and $z_k \in V_0 \oplus V_1$ for even $k$.

We take the norm on $X$ to be given by 
\[
    \| \gamma \| = |z_0|^2 + \frac{\pi}{2} \sum_{k\in \Z} |k|\vert z_k \vert^2.
\]
\end{defn}

\begin{lem}
\label{lem:Hilbert space and compact embedding}
For each $s\geq 0$, $H^{s}_{n,k}([0,1])$ is a Hilbert space with the inner product
\begin{equation*}
\langle \phi,\psi \rangle_{(s,n,k)} = \langle a_0,a'_0 \rangle + \frac{\pi}{2} \sum_{k\neq 0} 
|k|^{2s} \left \langle  a_k,a'_k \right \rangle.
\end{equation*} 
Furthermore, if $s > t$, then
the inclusion of $H^{s}_{n,k}([0,1])$ into $H^{t}_{n,k}([0,1])$ is compact.

In particular, $(X, \| \cdot \| )$ is a Hilbert space.
\end{lem}
\begin{proof}
Recall that $H^{s}(S^1, \R^{2n})$ is a Hilbert space. The involution on $\R^{2n}$ given by 
\[
\begin{split}
(x_1, \dots, x_n, y_1, \dots, &y_k, y_{k+1}, \dots, y_n) \mapsto \\
 &(x_1, \dots, x_n, y_1, \dots, y_k, -y_{k+1}, \dots, -y_n)
\end{split}
\]
induces an isometry on $H^{s}(S^1, \R^{2n})$ by acting on each Fourier coefficient. 
Observe now that $H^{s}_{n,k}([0,1])$ can be identified with the $+1$ eigenspace of this operator, and thus identifies
$H^{s}_{n,k}([0,1])$ as a closed subspace of a Hilbert space.
	
The compactness of the inclusion follows by considering the finite rank truncation operators 
\[
P_N \colon  \sum_{k} z_k \e^{k \pi J t} \mapsto \sum_{|k| \le N} z_k \e^{k \pi J t}.
\]
Let $\imath$ denote the inclusion $\imath \colon H^{s}_{n,k} \to H^t_{n,k}$. Then, in the operator norm
for $\imath$, $P_N \colon H^{s}_{n,k} \to H^t_{n,k}$, $ || P_N - \imath || \le C N^{t-s}$, and thus the inclusion is the uniform limit of finite rank operators, and is thus compact.
\end{proof}

\begin{lem}
Let $s > t$. If $j \colon H^{s}_{n,k}([0,1]) \to H^{t}_{n,k}([0,1])$ is the inclusion
operator, then the Hilbert space adjoint $j^{*}:H^{t}_{n,k}([0,1]) \to
H^{s}_{n,k}([0,1])$ is compact.
\qed
\end{lem}

\begin{lem}
\label{lem:Smoothness}
If $x \in H^{s}_{n,k}([0,1])$ for $s > \frac{1}{2} + r$, where $r$ is an integer,
then $x\in C_{n,k}^{r}([0,1])$.
\qed
\end{lem}

\begin{lem}
$j^{*}(L^2) \subset H^{1}$, and $\| j^*(y) \|_{H^1} \leq \|y\|_{L^2}$.
\qed
\end{lem}

\begin{defn} \label{def:projections}
The Hilbert space $X = H^{1/2}_{n,k}([0,1])$ admits a decomposition into negative, zero and positive Fourier frequencies:
\begin{align*}
X^-  &= \left \{ x \in H^{1/2}_{n,k}([0,1]) \, | \, x = \sum_{k < 0} x_k \e^{i \pi k t} \right \} \\
X^0 &= \left \{ x \in H^{1/2}_{n,k}([0,1]) \, | \, x = x_0 \in \R^{n,k} \right \} \\
X^+ &= \left \{ x \in H^{1/2}_{n,k}([0,1]) \, | \, x = \sum_{k > 0} x_k \e^{i \pi k t} \right \}
\end{align*}
Let $P^-, P^0$ and $P^+$ denote the orthogonal projections onto each of these subspaces, and we denote $x^{\pm} \coloneqq P^{\pm}(x)$ and $x^0 \coloneqq P^0(x)$.
\end{defn}

\subsection{An extended Hamiltonian}

Given a simple Hamiltonian $H\colon U \to \R$ with $m(H) > \frac{\pi}{2}$, 
we will analyze an associated Hamiltonian $\bar{H} \colon \R^{2n} \to \R$, and find a solution of $\dot{x}=X_{\bar{H}}(x)$ which is also a non-trivial solution of 
$\dot{x} = X_{H}(x)$. In the following, we construct 
the Hamiltonian $\bar{H}$.

We consider $n, k$ fixed and the simple Hamiltonian $H$ with $m(H) > \frac{\pi}{2}$ fixed.

\begin{defn}We now set some notation. 
    \label{def:qPi}

\begin{enumerate}
\item $\R^{2n}_{+} \ceq \{ z \in \R^{2n} \vert y_n > 0\}, R^{2n}_{-} \ceq
\{z \in \R^{2n} \vert y_n < 0\}$,

\item $U_\pm \coloneqq U \cap \R^{2n}_\pm$.

\item Let $q:\R^{2n}\to \R$ be the quadratic function
\[
q(x) = \,  \left(x_n^2 + y_n^2\right) +
 \frac{1}{N^2}\sum_{i=k+1}^{n-1}\left( x_i^2 +
y_i^2 \right) + \frac{2}{N^2}\sum_{i=1}^{k} (x_i^2 + y_i^2).
\]

Let $q_2:\R^{2n} \to \R$ be defined by
\[
q_2(x) = 
\begin{cases} 
	x_n^2 + y_n^2 & \text{for } y \ge 0 \\
	x_n^2 	        & \text{for } y < 0
\end{cases}
\]
and $q_{2n-2}:\R^{2n}\to \R$ be given by 
\[
    q_{2n-2}(x) = \frac{1}{N^2}\sum_{i=k+1}^{n-1}\left( x_i^2 + y_i^2 \right) + \frac{2}{N^2}\sum_{i=1}^{k} (x_i^2 + y_i^2).
\]

Define now 
\[q_\Pi(x) = q_2(x) + q_{2n-2}(x).\]

Choose $N$ sufficiently large so that 
\[
\supp dH \subset q_\Pi^{-1}( [0, 1)).
\]

Observe that $q_\Pi$ is a $C^1$ function with a jump discontinuity it its second
derivative.

\end{enumerate}
\end{defn}

Now, given a small $\epsilon> 0$ such that $\frac{\pi}{2}+\epsilon < m(H)$, we 
define $f:\R\to \R$ to be a function such that
\begin{align*}
&f(r) = m(H) \text{ for } r \leq 1 \\
&f(r) \geq \left (\frac{\pi}{2} + \epsilon \right )r \text{ for all } r\in \R \\
&f(r) = \left (\frac{\pi}{2} + \epsilon \right )r \text{ for } r \text{ large}\\
&0 < f'(r) \leq \left (\frac{\pi}{2} + \epsilon \right) \text{ for } r > 1.
\end{align*}
We define the extended Hamiltonian $\bar{H}$ by
\begin{equation}
\label{def:H extension}
\bar{H}(x) = \begin{cases}
H(x) & \text{ if } q_\Pi(x) \le 1\\
f( q_\Pi (x))  & \text{ if } q_\Pi(x) > 1.
\end{cases}
\end{equation}
In the next lemma, we give a criterion to show that certain orbits of the
Hamiltonian $\bar{H}$ are actually orbits of $H$.

\begin{lem} \label{lem:real solution}
Suppose $x(t), t\in [0,1]$ is a solution of $\dot{x} = X_{\bar{H}}$ 
such that 
$x(0),x(1) \in
\R^{n,k}$. If $\Phi_{\bar{H}}(x) > 0$, then $x(t)$ is 
non-constant and $x(t)$ is an orbit of $H$. 
\end{lem}
\begin{proof}
Let the functional $\Phi_{\bar{H}}\colon C_{n,k}^{\infty}\left([0,1]\right) \to \R$ be defined by Equation \ref{eq:Phi_H}. Note first that if $x$ is constant, then $\Phi_{\bar{H}}(x) \leq 0$, 
since $\bar{H} \geq 0$. 

To show the orbit of $\bar H$ is an orbit of $H$, we will show that $q_\Pi \le 1$ at each point of the orbit.
We will show instead that a chord $x(t)$ of $\bar H$ for which there exists a time
at which $q_\Pi(x(t)) > 1$ must have negative action.

Let $x(t)$ be such a trajectory, with $x(0), x(1) \in \R^{n,k}$ and with $q_\Pi(x(t)) > 1$ for some
time $t$. Notice that by construction, 
the region $\{ x \in \R^{2n} \, | \, q_\Pi(x) > 1 \}$ is flow invariant. 
Thus, the trajectory $x(t)$ has $q_\Pi(x(t)) > 1$ for all time.

We will first argue that any such trajectory must lie in the upper half-space
$\{ (x_1, \dots, x_n, y_1, \dots, y_n) \, | \, y_n \ge 0 \}$.
Indeed, since $q_\Pi(x(t)) > 1$, we have that the Hamiltonian vector field is
explicitly given by 
\[
    \dot x(t) = f'( q_\Pi(x(t))) J \nabla q_\Pi(x(t)).
\]
For all times $t$ at which $y_n < 0$, we have
\[
    \dot x_n(t) = 0 \qquad \dot y_n(t) = 2f'(q_\Pi(x(t))) x_n.
\]
In particular, $x_n$ is constant and $y_n$ is either monotone non-increasing or
monotone non-decreasing, depending on the sign of $x_n$. In particular then, it
is impossible for both $y_n(0) = 0$ and $y_n(1) = 0$ if there is a time $0 < t <
1$ at which $y_n(t) < 0$.  
The claim that the chord must lie in the upper half-space now follows.

Now, observe that on the upper half-space, we have $q_\Pi(x) = q(x)$, and hence
the Hamiltonian vector field on $\R^{2n}_+\backslash U_+$ is given by $X_{\bar H} = f'(q(x)) J \nabla q(x)$,
and thus $q(x)$ is an integral of motion in this region. 
It follows that $q(x(t)) = \tau > 1$ for all $t \in [0,1]$.  
Also notice that since $q(x)$ is
quadratic, we have $\langle x, \nabla q(x) \rangle = 2 q(x)$.
From this, we obtain:
\begin{align*}
\Phi_{\bar{H}}(x) &= \int_0^1 -\frac{1}{2} \langle J\dot{x}, x \rangle 
- \bar{H}(x(t)) \, dt \\
&= \int_0^1 \frac{1}{2}f'(q(x(t))) \langle \nabla q(x),x \rangle - f(\tau) \, dt \\
&= \int_0^1 f'(\tau) q(x(t))- f(\tau) \, dt\\
&= f'(\tau)\tau- f(\tau)\\
&\leq 0
\end{align*}
which completes the proof.
\end{proof}
\subsection{The action functional}
\begin{defn} \label{def:aa}
For $\phi,\psi  \in C_{n,k}^{\infty}([0,1])$, we define
\begin{align*}
a(\phi,\psi) =\frac{1}{2}  \int_0^1 \langle -J\dot{\phi},\psi \rangle \, dt.
\end{align*}
\end{defn}

We show the following simple lemma.
\begin{lem}
\label{lem:Pseudo-orthogonality}
For any $e_i = (0,\dots,0,1,0,\dots,0) \in \R^{2n}$, $i \in \{1,\dots,2n\}$,
\[
\int_0^1 \left \langle e^{k \pi J t}e_i, e^{l \pi J t}e_i \right \rangle \, dt =
\delta_{kl}
\]
\end{lem}
\begin{proof}
First, note that, if $0\leq i\leq n$, 
\[
e^{k \pi J t}e_i = (0,\dots,0,\cos(k \pi t),0,\dots,0,\sin(k \pi t),0,\dots,0),
\]
and if $n+1 \leq i \leq 2n$, then
\[
e^{k \pi J t}e_i = (0,\dots,0,-\sin(k \pi t),0,\dots,0,\cos(k \pi t),0,\dots,0).
\]
In either case, we have
\begin{align*}
\int_0^1 \left \langle e^{k \pi J t}e_i, e^{l \pi J t}e_i \right \rangle \, dt &= 
\int_0^1 \cos(k\pi t)\cos(l\pi t) + \sin(k\pi t)\sin(l \pi t) \, dt \\
&= \int_0^1 \cos((k-l) \pi t) \, dt \\
&= \delta_{kl}.
\end{align*}
\end{proof}

\begin{lem}
\label{lem:a}
For $\phi, \psi \in C^{\infty}_{n,k}([0,1])$, 
\begin{equation}
\label{eq:a}
a(\phi,\psi) = \frac{\pi}{2}\sum_{k > 0} |k|\langle z_k, w_k \rangle - \frac{\pi}{2} \sum_{k<0} |k| \langle z_k,w_k \rangle 
\end{equation}
where
\begin{equation}
\label{eq:phipsi}
\phi = \sum_{k\in \Z} z_k e^{k\pi J t} , \qquad \text{ and } \qquad 
\psi = \sum_{k\in \Z} w_k e^{k\pi J t}.
\end{equation}
\end{lem}
\begin{proof}
First, recall that, by Lemma \ref{lem:Fourier decomposition}, that for $\phi, \psi \in C^{\infty}_{n,k}$, the Fourier expansions
 $\phi = \sum_{k\in \Z} z_k e^{k\pi J t}$ and $\psi = \sum_{k\in \Z} w_k e^{k\pi J t}$ have that 
 $z_k, w_k \in V_0$ for odd $k$ and $z_k, w_k \in V_0 \oplus V_1$ for even $k$.

Substituting Equations \ref{eq:phipsi} into the expression for $a$ and using Lemma \ref{lem:Pseudo-orthogonality}, we get
\begin{align*}
a(\phi,\psi) &= \frac{1}{2} \sum_k k \pi \langle z_k, w_k \rangle \\
&= \frac{\pi}{2} \left ( \sum_{k > 0} |k| \langle z_k, w_k \rangle  - \sum_{k < 0} |k| \langle z_k, w_k \rangle \right ).
\end{align*}

\end{proof}
\begin{defn}
\label{def:a}
Given $\phi,\psi \in H^{1/2}_{n,k}([0,1])$, we define $a(\phi,\psi)$ by Equation \ref{eq:a},
and $a(\phi) \ceq a(\phi,\phi)$.
\end{defn}

\begin{rem} \label{rem:derivative of a}
Lemma \ref{lem:a} gives that Definitions \ref{def:a} and \ref{def:aa} are consistent, i.e.~they coincide for smooth 
paths, $\phi,\psi \in C^{\infty}_{n,k}([0,1])$. 
Recalling the norm on $X$ given in Definition \ref{def:X Hilbert space with
norm}, 
the function $a \colon X \to \R$ given by 
\[ a(\phi) = \| \phi^+ \|^2 -  \| \phi^- \|^2 \]
is therefore differentiable with derivative
\[ da(\phi)(\psi) = \langle (P^+ - P^-)\phi,\psi \rangle \]
and therefore the gradient $\nabla a$ is
\[ \nabla a(\phi) = (P^+ - P^-)\phi = \phi^+ - \phi^- \in X. \]
\end{rem}

For $\phi \in C^{\infty}_{n,k}([0,1])$, consider the expression 
\begin{equation*}
b(\phi) = \int_0^1 \bar{H}(\phi(t))\, dt.
\end{equation*}

Since, by construction, $|\bar{H}(x)| \leq M|x|^2$ for $q_{\Pi}(x)$ large, we have that $b$ 
may be extended to $L^2$, and therefore also on $H^{1/2} \subset L^2$. The following
results follow immediately from the proofs in \cite{Hofer_Zehnder_1994}.

\begin{lem}[\cite{Hofer_Zehnder_1994}, Section 3.3, Lemma 4]
\label{lem:b}
The map $b: X \to \R$ is differentiable. Its gradient is continuous and maps
bounded sets into relatively compact sets. Moreover,
\begin{equation*}
\| \nabla b(x) - \nabla b(y) \| \leq M \| x-y \|
\end{equation*}
and $\lvert b(x) \rvert \leq M \|x\|^2_{L^2_{n,k}}$ for all $ x,y \in X$. $\qed$
\end{lem}

\begin{rem}We now see that the functional $\Phi_{\bar{H}}:H^{1/2}_{n,k}([0,1]) \to \R$ given by 
\begin{equation*}
\Phi_{\bar{H}}(x) = a(x) - b(x)
\end{equation*}
is well-defined.  Furthermore, since $\bar{H}\in C^1([0,1],\R^{2n})$ and $a$ and $b$ are differentiable, $\Phi_{\bar{H}}$ is differentiable with gradient
\[ \nabla \Phi_{\bar{H}}(x) = x^+ - x^- - \nabla b(x). \qed \]
\end{rem} 

The results below summarize some of the properties of $\Phi_{\bar{H}}$ that we will
use in the following sections.
The proofs follow those given in
\cite{Hofer_Zehnder_1994}. Let $S = \{ (x_1, \dots, y_n) \, |\, -1 \le
y_n \le 1 \}$. 

\begin{lem}
\label{lem:C1}
Assume $x\in X$ is a critical point of $\Phi_{\bar{H}}$, i.e. $\nabla \Phi_{\bar{H}}(x) = 0$. Then $x$
is in $C_{n,k}^{1}([0,1])$. If, in addition, $x(t) \in \R^{2n}_+ \cup
\mathring{S}$ for all $t\in (0,1)$, then $x \in C^{\infty}_{n,k}([0,1])$.
\end{lem}

\begin{proof}
The proof given in Hofer and Zehnder \cite{Hofer_Zehnder_1994}, Section 3.3, Lemma 5 also applies in
this case. That is, we write $x$ and $\nabla (\bar{H}(x)) \in L^2_{n,k}$ by their Fourier
series, we have
\begin{align*}
x &= \sum_k e^{k\pi J t}x_k \\
\nabla \bar{H}(x) &= \sum_k e^{k\pi J t}a_k.
\end{align*}
Since $d\Phi_H(x)(v) = 0$, this implies that
\[
\left \langle (P^{+} - P^{-})x,v \right \rangle_{1/2,n,k} - \int_0^1 \left \langle 
\nabla \bar{H}(x(t)),v(t)
\right \rangle dt =0, \; \forall v \in X.
\]
Substituting the Fourier series of $x$ and $\nabla \bar{H}(x)$ into this expression, we obtain
\[
k \pi x_k = a_k.
\]
Therefore $a_0 = 0$ and
\[
\sum_k |k|^2|x_k|^2 \leq \sum |a_k|^2 < \infty.
\]
We conclude that $x \in H^1_{n,k}([0,1])$, and therefore
$x \in C^0_{n,k}([0,1])$ by Lemma \ref{lem:Smoothness}. It follows that
$\nabla \bar{H}(x(t)) \in C^0_{n,k}([0,1])$, so
\[
\xi(t) = \int_0^t J\nabla \bar{H}(x(s)) \, ds \in C^1(\R).
\]
However, it follows from the Fourier expansions that $\xi(t) = x(t) - x(0)$, and
therefore $x \in C^1([0,1])$ and solves
\[
\dot{x}(t) = J\nabla \bar{H}(x(t)).
\]
If $x(t) \in \bar{\R}^{2n}_{+}\cup S$ for all $t$, then $J\nabla \bar{H}(x(t)) \in C_{n,k}^1([0,1])$, so $x \in C_{n,k}^2([0,1])$. Repeating this, the second part of the lemma follows.
\end{proof}

\begin{lem}
\label{lem:Palais-Smale}
$\Phi_{\bar{H}}$ satisfies the Palais-Smale condition.
\end{lem}
\begin{proof}
We recall that, for $\Phi_{\bar{H}}$ to satisfy the Palais-Smale condition, we must have that, for every sequence $\{x_n\}$ with $\nabla\Phi_{\bar{H}}(x_n) \to 0$, there exists a convergent subsequence. If $\|x_n\|$ is bounded, then this follows from the compactness of $\nabla b$ and of $P^0$.

We now assume that the sequence of norms $\|x_n\|$ is unbounded. Consider the rescaled paths $y_n \ceq \frac{1}{\|x_n\|}x_n$, so that $\|y_n\| = 1$. Now, by assumption, 
\begin{equation*}
(P^+ - P^-)y_k - j^*\left(\frac{1}{\|x_k\|} \nabla \bar{H}(x_k)\right) \to 0.
\end{equation*}
Now note that there exists an $M$ such that $|\nabla \bar{H}(z)| < M|z|$ for all $z \in \R^{2n}$.  It follows that the sequence
\begin{equation*}
\frac{\nabla \bar{H}(x_k)}{\|x_k\|} \in L^2
\end{equation*}
is bounded in $L^2$.

Since $j^*:L^2\to X$ is compact, $(P^+ - P^-)y_k$ is relatively compact, and $y^0_k$ is bounded in $\R^{2n}$, it follows that the sequence $y_k$ is relatively compact in $X$. Let $\epsilon>0$ be as in the definition of $\bar{H}$ in Equation \ref{def:H extension}. 
Define
\begin{equation*}
Q(x) = \left ( \frac{\pi}{2} + \epsilon \right ) q_\Pi(x).
\end{equation*}
After taking a subsequence we may assume that $y_k \to y$ in $X$ and therefore $y_k \to y$ in $L^2$. Note that, since $\nabla Q$ defines a continuous operator on $L^2$, and also that, for $\lambda > 0$,
\[
\nabla Q(\lambda x) = \lambda \nabla Q(x).
\] 
It follows that
\begin{equation*}
\begin{split}
\left\Vert \frac{ \nabla \bar{H}(x_k) }{\left\Vert x_k \right\Vert} - \nabla Q(y) \right\Vert_{L^2} \leq & \left\Vert \frac{\nabla \bar{H}(x_k)}{\left\Vert x_k \right\Vert} - \nabla Q(y_k)\right\Vert_{L^2} \\
&+ \left\Vert \nabla Q(y_k) - \nabla Q(y) \right\Vert_{L^2}\\
& =  \frac{1}{\left\Vert x_k \right\Vert}\left\Vert \nabla \bar{H}(x_k)- \nabla Q(x_k)\right\Vert_{L^2} \\
&+ \left\Vert \nabla Q(y_k) - \nabla Q(y) \right\Vert_{L^2}.
\end{split}
\end{equation*}
Since, furthermore, $\vert \nabla \bar{H}(z) - \nabla Q(z) \vert \, \leq M$ for all $z\in \R^{2n}$, we may conclude that
\begin{equation*}
\frac{\nabla \bar{H}(x_k)}{\left \Vert x_k \right \Vert } \to \nabla Q(y) \text{ in } L^2.
\end{equation*}
Therefore,
\[
\frac{\nabla b(x_k)}{\left\Vert x_k \right\Vert} = j^*\left( \frac{\nabla \bar{H}(x_k)}{\left \Vert x_k \right \Vert }\right) \to
j^*\left(\nabla Q(y) \right) \text{ in } X.
\]
It follows from this convergence that $y$ satisfies the following system of equations in $X$:
\begin{align*}
y^+ - y^-  - j^* \nabla Q(y) &= 0,\\
\left \Vert y \right \Vert &= 1.
\end{align*}
As in Lemma \ref{lem:C1}, we now have that $y \in C^1([0,1],\R^{2n})$ and that $y$ also satisfies the Hamiltonian equation
\begin{equation}
\label{eq:Hamiltonian system}
\begin{split}
&\dot{y}(t) = X_{Q}(y(t)),\\
&y(0), y(1) \in \R^{n,k}.
\end{split}
\end{equation}
By construction of $Q$, however, there are no non-trivial solutions of \eqref{eq:Hamiltonian system}. 
This, however, contradicts the assumption that $\Vert y \Vert= 1$, and we conclude that the sequence $x_k$ must be bounded, proving the lemma.
 
\end{proof}

\begin{lem}
\label{lem:Global flow}
The equation
\[
\dot{x} = -\nabla\Phi_{\bar{H}}(x), \; x\in X
\]
defines a unique global flow $\R \times X \to X:(t,x) \mapsto \phi^{t}(x) \equiv
x\cdot t$.
\qed
\end{lem}
\begin{proof}
This follows immediately from the global Lipschitz continuity of  
$\nabla \Phi_{\bar{H}}$ as a vector field on $X$.
\end{proof}

\begin{lem} \label{lem:FlowRepresentation}
The flow of the ODE $\dot x = -\nabla \Phi_{\bar{H}}(x)$ has the following form 
\begin{equation}
\phi^{t}(x) = e^t x^{-} + x^{0} + e^{-t}x^{+} + K(t,x),
\end{equation}
where $K:\R \times X \to X$ is continuous and maps bounded sets into precompact
sets
and $x^- = P^-(x)$, $x^0 = P^0(x)$ and $x^+ = P^+(x)$.
\end{lem}

\begin{proof}

The proof of this lemma follows exactly the proof in Hofer and Zehnder 
\cite{Hofer_Zehnder_1994}, Section 3.3, Lemma 7. The key point is that if we explicitly define $K$ 
by the formula 
\[
K(t,x) = -\int_0^t \left( e^{t-s}P^- + P^0 + e^{-t+s} P^+
\right)\nabla b(x \cdot s) 
\, ds,
\]
we may verify directly that this has the required properties.
\end{proof}

\subsection{Existence of a chord}

We will now complete the proof of Proposition \ref{prop:Upper Bound}. 
To do this, we will prove the following:

\begin{thm}
\label{thm:Upper bound}
If $H$ is a simple Hamiltonian on $(U, U^{n,k})$ and $m(H) > \frac{\pi}{2}$, then there  
exists an orbit of the system $\dot{x} = X_H(x)$ with return time $T =1$
and $\Phi_{\bar{H}}(x) > 0$.
\end{thm}

The remainder of this section will prove the theorem.  The proof follows closely 
the proof of \cite{Hofer_Zehnder_1994}, Section 3.1, Theorem 2, though it introduces some 
new subtleties.  
We start by recalling the Minimax Lemma (see \cite{Hofer_Zehnder_1994}, page 79 for a proof), 
which will play a key role.

\begin{defn}
	Let $f:X \to \R$ be a differentiable function on a Hilbert space $X$, i.e. $f\in C^1(X,\R)$, and let $\mathcal{F}$ be a family of subsets $F\subset X$. We call the value
	\[
	c(f,\mathcal{F}) \ceq \inf_{F \in \mathcal{F}} \sup_{x\in F} f(x) \in \R \cup \{\infty\} \cup \{-\infty\}
	\]
	the \defin{minimax} of $f$ on the family $\mathcal{F}$.
\end{defn}

\begin{lem}[Minimax Lemma]
	Suppose $f \in C^1(X,\R)$, where $X$ is a Hilbert space, and that $f$ satisfies the following conditions:
	\begin{enumerate}
		\item $f$ is Palais-Smale,
		\item $x = -\nabla f(x)$ defines a global flow $\phi_t(x)$ on $X$,
		\item The family $\mathcal{F}$ is positively invariant under the flow, i.e., $\phi_t(F) \in \mathcal{F}$
		for all $F \in \mathcal{F}$ and all $t\geq 0$,
		\item $-\infty < c(f,\mathcal{F}) < \infty$,
	\end{enumerate}
	then the real number $c(f,\mathcal{F})$ is a critical value of $f$, that is, there exists an element $x^* \in X$
	with $\nabla f(x^*) = 0$ and $f(x^*)= c(f,\mathcal{F})$.
\end{lem}

We will use the Minimax Lemma above over the family of sets 
$\mathcal{F} = \{ \phi^{t}(\Sigma_\tau) \}$ to establish the existence 
of a critical point of the action functional. 
As established in Lemma \ref{lem:real solution}, 
it suffices to show this for the Hamiltonian $\bar{H}$, 
as the resulting orbit will be an orbit of $H$. 

The plan of the proof is as follows. 
In Lemmas \ref{lem:pointwise inequality} and \ref{lem:Integral inequality}, 
we prove a pair of technical inequalities on the polynomial part of $\bar{H}$. 
Then, we produce two ``half-infinite'' dimensional subsets of $X$,
$\Sigma$ and $\Gamma$, and in Lemmas \ref{lem:Linking 1} and \ref{lem:Linking 2} 
we show that the action $\Phi_{\bar{H}}|_{\partial\Sigma} 
< 0$ and that the action $\Phi_{\bar{H}}|_{\Gamma} > 0$, respectively. 
We then use the a Leray-Schauder degree argument in Lemma \ref{lem:Linking 3} 
to show that the flow of $\phi_t(\Sigma_{\tau})$ intersects $\Gamma_{\alpha}$ for 
all $t \geq 0$, and finally, we apply the Minimax Lemma to the union of the sets 
$\phi_t(\Sigma_{\tau})$, which proves the result.

We
begin with the
following lemma.

\begin{lem}
 \label{lem:H zero at origin}
Let $H\in \mathcal{H}(U,U^{n,k})$. Then there exists a compactly
supported Hamiltonian diffeomorphism
$\psi \colon U \to U$ with $\psi( U^{n,k} ) = U^{n,k}$
such that $H \circ \psi \in \mathcal{H}(U,U^{n,k})$ and $H \circ \psi$ 
vanishes in a neighbourhood of $0$.
\end{lem}

\begin{proof}
Observe that in order for a Hamiltonian $K$ to have a Hamiltonian vector field 
whose flow preserves $U^{n,k}$,
the following derivatives 
\[
\frac{\partial}{\partial x_i} K(x_1, \dots, x_k, x_{k+1}, \dots, x_n, y_1, \dots, y_k, 0, \dots, 0)  = 0 \qquad \text{ for } i \ge k+1
\]
must vanish along $U^{n,k}$.

By hypothesis, $H$ is admissible, so there exists an interior point $p \in U^{n,k}$ in whose neighbourhood $H$ vanishes.
Let $V$ be a neighbourhood of the ray $ \{ \tau p \, | \, \tau \in [0,1] \}$ that is invariant under the involution 
\begin{dmath}
c_{n,k} \colon (x_1, \dots, x_n, y_1, \dots, y_k, y_{k+1}, \dots, y_n) \mapsto
(x_1, \dots, x_n, y_1, \dots, y_k, -y_{k+1}, \dots, -y_n).
\end{dmath}
Let $\rho$ be a $c_{n,k}$-invariant cut-off function, identically equal to $1$ on the neighbourhood $V$ and whose support
is compactly contained in the interior of $U$.

Now define a Hamiltonian by $K \colon Z(1) \to \R$ by 
$$ K \colon z \mapsto \rho(z) \langle z, -J p \rangle.$$ Let $X_K$ be its associated Hamiltonian vector field and $\psi_K$ its time $1$ map.

Observe first that the Hamiltonian vector field $X_K(z) = p$ for any $z\in V$, so
$\psi_K(0) = p$ and thus $H \circ \psi_K$ vanishes in a neighbourhood of $0$.

A computation of $\partial_{x_j} K$ for $j \ge k+1$ shows that the vector field is tangent to $U^{n,k}$ (using 
both that $p \in U^{n,k}$ and that $\rho$ is $c_{n,k}$-invariant).
\end{proof}

From now on, without loss of generality, we assume that $H$ vanishes in a neighborood of $0$.

\begin{prop}
\label{prop:Existence of an orbit}
There exists $x^{*} \in X$ satisfying $\nabla 
\Phi_{\bar{H}}(x^{*}) = 0$ and $\Phi_{\bar{H}}(x^{*})
> 0$.
\end{prop}

The proof of Proposition \ref{prop:Existence of an orbit} follows from the following
lemmas. We set some notation for the discussion which follows.

\begin{defn}
\begin{enumerate}
 \item $e_n \ceq (0,\dots,x_n=1,0,\dots,0)^T$

 \item  $e^+(t) \ceq e^{\pi J t}e_n = (0, \dots,0, x_n= \cos(\pi t), 0, \dots, 0,
     y_n=\sin(\pi t))^T$
\item  \[\begin{split}
\Sigma_{\tau} \ceq \lbrace x \in X \, \vert \, 
x = x^{-} + x^{0} + se^{+} ,
x^{-} \in X^{-}, x_{0} \in
 X^{0} ,\\
\| x^{-} + x^{0}\| \,\leq \tau, \text{ and } 0 \le s \leq \tau \rbrace
\end{split}\]
\item  $\Gamma_{\alpha} \ceq \left\lbrace x \in X^+ \, \vert \, \| x\| = \alpha
\right\rbrace$

\end{enumerate}
\end{defn}

\begin{lem} \label{lem:pointwise inequality}
    Let $u =(0,\dots,0,\xi,0,\dots,0, \eta) \colon [0,1] \to \R^{2n}$ be a smooth function, where $\langle u(t),e_n \rangle = \xi(t)$ and $\langle u(t),e_{2n}\rangle = \eta(t)$ are the $x_n$ and $y_n$ coordinates, respectively, of $u(t)$, and suppose that $s \ge 0$. Then 
    \[
        q_2( u(t) + se^+(t) ) \ge s^2 + 2 s \langle e^+(t), u(t) \rangle +
        \xi(t)^2,
    \]
    where $q_2$ is as in Definition \ref{def:qPi}.
\end{lem}
\begin{proof}
    Recall that, for $x \in \R^{2n}$,
\[
q_2(x) = 
\begin{cases} 
	x_n^2 + y_n^2 & \text{for } y_n \ge 0 \\
	x_n^2 	        & \text{for } y_n < 0.
\end{cases}
\]

Let $\pi_n:\R^{2n} \to \R^{2}$ be given by $\pi_n(x) = (x_n,y_n)$. We now calculate
\[
q_2 (se^+ + u) = \begin{cases} s^2 + \langle 2se^+ , u \rangle + \xi^2(t) + \eta^2(t)
    &\text{if } \pi_n((se^+ + u)(t)) \in \R^{2}_+\\
    s^2 \cos^2 (\pi t) + 2s \cos (\pi t) \xi(t)  + \xi^2(t) \quad &\text{if } \pi_n((se^+ + u)(t)) \in \R^{2}_-
\end{cases}
\]

If $t$ is such that $\pi_n(se^+(t) + u(t)) \in \R^2_+,$ the result follows
immediately. We consider then the case when $\pi_n(se^+(t) + u(t)) \in \R^{2}_-$.
Equivalently, this occurs when $s \sin(\pi t) + \eta(t) \le 0$.

We compute
\begin{align*}
s^2 \cos^2 (\pi t) + 2s \cos (\pi t)  \xi(t)   = & \begin{aligned}[t] &s^2 \cos^2 (\pi t) + 2s \cos (\pi t)  \xi(t) \\
& + 2s \sin (\pi t) \eta(t) - 2s \sin (\pi t) \eta(t) \end{aligned}\\
=& s^2 \cos^2 (\pi t) + \langle 2s e^+ , u \rangle - 2s \sin (\pi t) \eta(t)\\
=&s^2 (1 - \sin^2 (\pi t)) + \langle 2s e^+ , u \rangle - 2s \sin (\pi t) \eta(t)\\
=&s^2 + \langle 2s e^+ , u \rangle - s \sin (\pi t) \left( s \sin (\pi t) + 2 \eta(t) \right).
\end{align*}

Observe now that we have $s \sin(\pi t) + \eta(t) \le 0$, but $t \in [0,1]$ and
$s \ge 0$, so it follows that $\eta(t) \le -s \sin(\pi t) \le 0$. 
Thus, $s \sin( \pi t) + 2 \eta(t) \le 0$, and hence:
\begin{align*}
q_2(x) 
= &s^2 + \langle 2s e^+ , u \rangle - s \sin (\pi t) \left( s \sin (\pi t) + 2
\eta(t) \right) + \xi^2 \\
\ge  s^2 + 2s \langle e^+(t), u(t) \rangle + \xi(t)^2,
\end{align*}
proving the result.
\end{proof}

\begin{lem}
\label{lem:Integral inequality} For $\tau > 0$ and $x = x^- + x^0 + se^+\in \Sigma_\tau$
\begin{align*}
 \int_0^1 q_\Pi(x) \, dt \geq
 \int_0^1 q_\Pi(x^0)\, dt +
\int_0^1 q_\Pi(se^+)\, dt.
\end{align*}
\end{lem}

\begin{proof}

Recall that 
$q_\Pi(x) = q_2(x) + q_{2n-2}(x)$,
where 
\[
    q_{2n-2}(x) = \frac{1}{N^2}\sum_{i=k+1}^{n-1}\left( x_i^2 + y_i^2 \right) + 
\frac{2}{N^2}\sum_{i=1}^{k} (x_i^2 + y_i^2)\]
and $q_2$ is as in Definition \ref{def:qPi}.

If $x_1$ and $x_2$ are in orthogonal subspaces of $L^2([0,1],\R^{2n})$\[
 \int_0^1 \langle x_1(t),  x_2(t)\rangle \, dt = 0,
\]
it follows that 
\begin{equation} \label{eqn:qn estimate}
 \int_0^1 q_{2n-2}(x) \, dt = \int_0^1 q_{2n-2}(x^-) \, dt 
+ \int_0^1 q_{2n-2}(x^0) \,dt 
+ \int_0^1 q_{2n-2}(x^+) \, dt.
\end{equation}

Now, consider a smooth element $x$ of $L^2_{n,k}([0,1])$ of the form $x = x^-+x^0+se^+$,
with $s \ge 0$, and $x^- \in X^-, x^0 \in X^0$. 
Let $\xi^-(t)$ be the projection of $x^-(t)$ to the $x_n$ coordinate,
and similarly let $\xi^0$ be the projection of $x^0$. Then, $\xi(t) = \xi^-(t) +
\xi^0$ is the projection of $x^-(t)+x^0$. Note that by Lemma \ref{lem:Fourier
decomposition}, we have $\xi^0 = a_0e_n$ and 
\[
    \xi^-(t) = \sum_{k < 0} a_k \cos( k \pi t),
\]
where the real constants $a_0$, $a_k, k < 0$ are obtained as the projections to
$e_n$ of the terms $z_k$ as given in Lemma \ref{lem:Fourier decomposition}.

By Lemma \ref{lem:pointwise inequality} and using the fact that $x^-+x^0$
is orthogonal to $e^+$, we have
\begin{align*}
    \int_0^1 q_{2}(x) \, dt &\geq \int s^2 + 2 s \langle e^+, x^-+x^0 \rangle + \xi^2
    \, dt\\
    &= \int_0^1 q_2(s e^+) \, dt + \int_{0}^1 \xi^2 \, dt.
\end{align*}

Now, we observe that $q_2(x^0) = (\xi^0)^2$, since $x^0 \in V_0 \cap V_1$, and therefore
\begin{align*}
    \int_0^1 \xi^2 \, dt &= \int_0^1 (\xi^0)^2 \, dt + \int_0^1 (\xi^-)^2 \, dt \\
    &\ge \int_0^1 (\xi^0)^2 \\
    &= \int_0^1 q_2(x^0) \, dt.
\end{align*}

It now follows that 
\begin{equation} \label{eqn:q2 estimate}
    \begin{aligned}
    \int_0^1 q_2(x) \, dt &\ge \int_0^1 q_2(s e^+) \, dt + \int_{0}^1 \xi^2 \,
    dt \\
     &\ge \int_0^1 q_2(s e^+) \, dt + \int_0^1 q_2(x^0) \, dt.
\end{aligned}
\end{equation}

Combining now the inequalities \eqref{eqn:qn estimate} and
\eqref{eqn:q2 estimate}, we obtain for smooth $x = x^- + x^0 + s e^+$:
\[
    \int_0^1 q_\Pi(x) \, dt \ge \int_0^1 q_\Pi(se^+) \, dt + \int_0^1 q_\Pi(
    x^0) \, dt.
\]
It now follows by continuity for all $x = x^- + x^0 + se^+ \in L^2_{n,k}$.

\end{proof}

\begin{lem}
\label{lem:Linking 1}
There exists a $\tau^* > 0$ such that for $\tau > \tau^*$,
 \begin{equation*}
  \Phi_{\bar{H}}|_{\partial\Sigma_{\tau}} \leq 0.
 \end{equation*}
\end{lem}

\begin{proof}
 First, recall that $\Phi_{\bar{H}}(x) = a(x) + b(x)$. Since $b
 \leq 0$ and 
$a |_{ X^{-}\oplus X^{0}} \leq 0$ 
we have that $\Phi_{\bar{H}} |_{ X^{-} \oplus X^{0} }\leq 0$. We now
need to examine $\Phi_{\bar{H}}$ on the boundary regions, where either $\| x^- + x^0 \| =
\tau$ or $s = \tau$. We note that by the construction of
$\bar{H}$ above, there exists a constant $C>0$ such that
\begin{equation*}
\bar{H}(z) \geq \left(\frac{\pi}{2} + \epsilon \right )q_\Pi(z) - C \quad \forall z \in \R^{2n}.
\end{equation*}
Therefore,
\[
 \Phi_{\bar{H}}(x) \leq a(x) - \left(\frac{\pi}{2} + \epsilon\right)\int_0^1 q_\Pi(x(t)) \, dt + C \quad \forall x \in X.
\]
We now estimate $\Phi_{\bar{H}}(x)$ for $x(t) = x^-(t) + x^0 + se^+(t)$ with $s
\ge 0$. 
Note that by Lemma \ref{lem:Fourier decomposition}, $x^0 \in \R^{2n}_+$.  
Lemma \ref{lem:Integral inequality} gives
\begin{align*}
 \Phi_{\bar{H}}(x^- + &x^0 + se^+) \\
                      & \leq a(x^-+x^2+se^+) - 
 \left (\frac{\pi}{2} + \epsilon \right )\int_{0}^{1} q_\Pi(se^+(t)) + q_\Pi(x^0)\,
 dt + C\\
 \intertext{Using now Definition \ref{def:a} and Remark \ref{rem:derivative of
 a}:}
 & \leq  
 s^2\|e^+\|^2 - \|x^-\|^2 - \left (\frac{\pi}{2} + \epsilon \right )\int_{0}^{1}
 q_\Pi(se^+(t)) + q_\Pi(x^0)\, dt
 + C\\
 &= \ C 
 + s^2 \| e_+ \|^2 -\|x^-\|^2 - \left ( \frac{\pi}{2} + \epsilon \right) q_\Pi(x^0) 
 - s^2\left (\frac{\pi}{2} +\epsilon \right ) \int_{0}^{1} q_\Pi(e^+(t)) \, dt. 
\end{align*}
Recalling the definition of the norm from Definition \ref{def:X Hilbert space
with norm},
$|| e^+ ||^2 = \frac{\pi}{2}$, $\int_0^1 q_\Pi(e^+) dt = 1$, and $q_\Pi(x^0)
= \|x^0\|^2$, it follows that
\[
\Phi_{\bar{H}}(x^- + x^0 + se^+) \leq 
	C -  \Vert x^-\Vert^2 - \left ( \frac{\pi}{2} + \epsilon \right) \Vert x^0\Vert^2 - \epsilon s^2,
\]
and thus there is a $\tau >0$, such that $\Phi_{\bar{H}}(x)|_{\partial\Sigma_{\tau}} \leq 0$.
\end{proof}

\begin{lem}
\label{lem:Linking 2} There exists $\alpha$ and $\beta$ such that  
$\Phi_{\bar{H}}|_{\Gamma_{\alpha}} \geq \beta
> 0$
\end{lem}

\begin{proof}
The proof proceeds exactly as in \cite{Hofer_Zehnder_1994}, Section 3.4, Lemma 9. As they observe, this lemma follows from the Sobolev inequality $\Vert u \Vert_{L^3} \leq C \Vert u \Vert_{1/2}$. Since $\bar{H}$ vanishes at the origin, Taylor's theorem and the fact that $\bar{H}$ is quadratic at infinity implies that we may find a constant $K >0$ such that $\vert \bar{H} \vert \leq K \vert x \vert^3$, and therefore
\[
\Phi_{\bar{H}}(x) \geq \frac{1}{2}\Vert x^+ \Vert^2 - \frac{1}{2}\Vert x^- \Vert^2 - CK\Vert x \Vert^3.
\]
For $x \in X^+$ with $\Vert x \Vert$ sufficiently small, the result follows.
\end{proof}

\begin{lem}
\label{lem:Linking 3}
 $\phi^{t}(\Sigma_\tau) \cap \Gamma_\alpha \neq \emptyset$, for all $t \geq 0$.
\end{lem}

\begin{proof}
The proof of this lemma proceeds as in \cite{Hofer_Zehnder_1994}, Section 3.4, Lemma 10, which we summarize here. We use the Leray-Schauder degree to show the existence of an element in
$\phi^{t}(\Sigma) \cap \Gamma$. 
(See Deimling \cite{Deimling_1985}, Theorem 8.2 or Zeidler \cite{Zeidler_1986}, Chapter 12, for properties of the Leray-Schauder degree.) 
Let $F$ denote the space $X^{-} + X^{0} + \R e^{+}$. Using the expression in Lemma \ref{lem:FlowRepresentation}, we will rewrite the condition
\begin{equation}
\label{eq:Condition 4}
\phi^{t}(\Sigma_{\tau}) \cap \Gamma_{\alpha} \neq \emptyset
\end{equation} 
in the form $x + B(t,x) = 0$ for the operator $B:\R \times F \to F$ defined by
\[
B(t,x) \ceq (e^{-t}P^{-} + P^{0})K(t,x) + P^{+}\left((\Vert \phi^{t}(x) \Vert - \alpha)e^{+}
-x \right).
\] 
We remark that $B$ is continuous and maps bounded sets into relatively compact sets by Lemma \ref{lem:FlowRepresentation}.
We now recall that, since $x\in \Sigma_{\tau}$, 
$x=x^{-} + x^{0} + se^{+}$, for some $0 \leq s \leq \tau$, so the system of Equations
\ref{eq:Condition 4} is equivalent to
\begin{equation}
\label{eq:L-Sh}
\begin{split}
0 =&\, x + B(t,x)\\
x \in&\, \Sigma_\tau.
\end{split}
\end{equation}

Let $I$ denote the identity operator. By the Leray-Schauder degree theory, for any
fixed $t\geq 0$, Equation \ref{eq:L-Sh} has a solution $x\in \Sigma_{\tau}$ if
\begin{equation*}
\text{deg}(\Sigma_{\tau}, I + B(t,\cdot),0) \neq 0.
\end{equation*}

Since, by Lemmas \ref{lem:Linking 1} and \ref{lem:Linking 2}, 
$\phi^{t}(\partial \Sigma_{\tau}) \cap \Gamma =
\emptyset$ for $t \geq 0$, there is no solution of Equation \ref{eq:L-Sh} on the
boundary
$\partial \Sigma_{\tau}$. Therefore, since the Leray-Schauder degree is homotopy
invariant, we have
\begin{equation*}
\text{deg}(\Sigma_{\tau}, I +B(t,\cdot),0) = \text{deg}(\Sigma_{\tau},I +
B(0,\cdot),0).
\end{equation*}
We see that $K(0,x) = 0$, so $B(0,x) = P^{+}\left((\Vert x \Vert -\alpha)e^{+}
-x\right)$.
We define $h:[0,1] \times X \to X^{+}$ by
\begin{equation*}
h(\mu,x) = P^{+}\left((\mu \|x\| -\alpha ) e^{+} - \mu x \right),
\end{equation*}
and we claim that $x + h(\mu,x) \neq 0$ for $x\in \partial \Sigma_{\tau}$.

To see this, note first that if $x\in \Sigma_{\tau}$ solves $x + h(\mu,x)= 0$
then $x = se^{+}$, so $s((1-\mu) + \mu\| e^{+} \|) = \alpha$. 
Therefore, $0 < s \leq
\alpha$, so $x\notin \partial \Sigma_{\tau}$ if $\tau > \alpha$, which is true by
hypothesis. Furthermore, since $\tau > \alpha$, $\alpha e^{+} \in 
\Sigma_{\tau}$,
so by homotopy,
\begin{align*}
\deg(\Sigma_{\tau},I + B(t,\cdot),0) &= \deg(\Sigma_{\tau},I + h(0,\cdot),0) \\
&= \deg(\Sigma_{\tau},I - \alpha e^{+},0)\\
&= \deg(\Sigma_{\tau},I,\alpha e^{+}) \\
&= 1.
\end{align*}
This completes the proof.
\end{proof}

We now proceed with the proof of Proposition \ref{prop:Existence of an orbit}.

\begin{proof}[Proof of Proposition \ref{prop:Existence of an orbit}]
Let $\alpha$ be such that $\Sigma_{\tau}$ and $\Gamma_{\alpha}$ satisfy the
hypotheses of Lemmas \ref{lem:Linking 1} and \ref{lem:Linking 2}. 
Let $\mathcal{U}$ be the union
\begin{equation*}
\mathcal{U} \ceq \bigcup_{t \geq 0} \phi^t(\Sigma_{\tau}),
\end{equation*}
and define 
\begin{equation*}
c(\Phi_{\bar{H}},\mathcal{U}) \ceq \inf_{t\geq 0} \sup_{x \in \phi^t{\Sigma_{\tau}}}
\Phi_{\bar{H}}(x).
\end{equation*}\
We wish to apply the Minimax Lemma to $\Phi_{\bar{H}}$ and
$c(\Phi_{\bar{H}},\mathcal{U})$.

We first check that $c(\Phi_{\bar{H}},\mathcal{U})$ is finite. Since, by Lemmas \ref{lem:Linking 1}, \ref{lem:Linking 2}, and
the hypothesis on $\alpha$, we have $\phi^t(\Sigma_{\tau}) \cap \Gamma_{\alpha}
\neq \emptyset$ and $\Phi_{\bar{H}}|_{\Gamma_{\alpha}} \geq \beta$, we have
\begin{equation}
\label{eq:Ineq 1}
\beta \leq \inf_{x \in \Gamma_\alpha} \Phi_{\bar{H}}(x) \leq \sup_{x \in 
\phi^t(\Sigma_{\tau})}
\Phi_{\bar{H}}(x).
\end{equation}
By Lemma \ref{lem:b}, $\Phi_{\bar{H}}$ maps bounded sets into bounded sets. Therefore, for each
$t\geq 0$,
\begin{equation}
\label{eq:Ineq 2}
\sup_{x\in \phi^t(\Sigma_{\tau})} \Phi_{\bar{H}}(x) < \infty.
\end{equation}
Combining the inequalitites \ref{eq:Ineq 1} and \ref{eq:Ineq 2} we see that for
every $t\geq 0$,
\begin{equation*}
-\infty < \beta < \sup_{x \in \phi^t(\Sigma_{\tau})} \Phi_{\bar{H}}(x) < \infty
\end{equation*}
and therefore $-\infty < c(\Phi_{\bar{H}},\mathcal{U}) < \infty$.
By Lemma \ref{lem:Palais-Smale}, $\Phi_{\bar{H}}$ satisfies the Palais-Smale condition, 
and by Lemma \ref{lem:Global flow},
the equation $\dot{x} = \nabla \Phi_{\bar{H}}(x)$ generates a global flow, from which it follows that 
$\phi^t(\mathcal{U}) \subseteq \mathcal{U}$. By the Minimax Lemma,
$c(\Phi_{\bar{H}},\mathcal{U})$ is a critical value. There is therefore a point $x^* \in
X$
with $\nabla \Phi_{\bar{H}}(x^*) = 0$ and $\Phi_{\bar{H}}(x^*) = 
c(\Phi_{\bar{H}},\mathcal{U}) \geq \beta
> 0$, which completes the proof.
\end{proof}

Theorem \ref{thm:Upper bound} now follows immediately.

\section{Existence of chords near an energy surface}
\label{sec:Almost existence}

We give here a dynamical consequence of our constructions: that 
the existence of the capacity ${c}$ proven in Theorem \ref{thm:Capacity} implies the existence of Hamiltonian chords on a large family of energy surfaces.

\begin{defn}
Let $H \colon M \to \R$ be a Hamiltonian function on the symplectic manifold
$(M,\omega)$ and $\lambda \in \R$. We call  $S = H^{-1}(\lambda)$ a \emph{regular
energy surface with energy $\lambda$} if $dH(x) \neq 0$ for $x\in S$. 
\end{defn}

\ThmAlmostExistence

\begin{proof}
The proof follows closely the proof of Theorem 1 in Chapter 4 of \cite{Hofer_Zehnder_1994}. 
The new ingredient here comes from the fact that the admissible Hamiltonians in
the coisotropic setting require that trajectories either be constant or have
positive return time (i.e.~ruling out trajectories that have tangencies to the
isotropic leaves). This will be dealt with by Lemma \ref{lem:Transversality and
admissibility} below.

Denote level sets by $S_\lambda = H^{-1}(\lambda)$. Since $S_1 \subset U$, and since
transversality is an open condition, there
exists a $\rho > 0$ such that for every energy $\lambda \in (1-\rho, 1+\rho)$, 
$S_\lambda \subset U$ and $S_\lambda$ is transverse to $N$.

By shrinking $U$ as necessary, we may assume $U = H^{-1}(1-\rho, 1+\rho)$. 
Monotonicity of the capacity gives that the smaller $U$ also has finite capacity.

We will construct an auxiliary Hamiltonian function $F$ on $U$ which is constant on every surface $S_\lambda$ contained in $U$. 
Choose $\epsilon$ in $(0,\rho)$, and let $f:\R \to \R$ be a smooth function such that
\begin{align*}
& f(s) = {c}(U,N,\omega, \sim) + 1 & \text{ for } & s \leq 1 - \epsilon \text{ and } s \geq 1 + \epsilon \\
& f(s) = 0  & \text{ for } & 1 - \frac{\epsilon}{2} \leq s \leq 1 + \frac{\epsilon}{2} \\
& f'(s) < 0 & \text{ for } & 1 - \epsilon < s < 1 - \frac{\epsilon}{2} \\
& f'(s) > 0 &  \text{ for } & 1 + \frac{\epsilon}{2} < s < 1 + \epsilon.
\end{align*}
Define $F \colon U \to \R$ by $F(x) \colonequals f\left(H(x)\right)$ for $x \in
U$, and extend $F$ to $F \colon M \to \R$ by defining $F(x) \colonequals
{c}(U,N,\omega, \sim)+1$ for $x\in M\backslash U$.

Observe that this function $F$ is therefore simple (see Definition \ref{def:simple}). 
The maximum of $F$, $m(F) > {c}(U, N, \omega, \sim)$, so $F$ cannot be 
admissible. The failure of admissibility either gives the existence of a short
leafwise chord of $F$ or there is a non-constant trajectory that fails to leave its isotropic leaf. 
We use the following lemma to rule out the latter case:
\begin{lem}
\label{lem:Transversality and admissibility}
Let $N \subset M$ be a coisotropic submanifold and $H \colon M \to \R$ be a function.
If $x \in N$ satisfies that $T_x N + \ker dH_x = T_x M$, then if $X_H(x)$ is tangent to 
the isotropic leaf through $x$, then $X_H(x) = 0$.

\end{lem}

\begin{proof}
Let $K$ denote the isotropic leaf through $x$. If $X_H(x) \in T_x K$, we then
have for any $v \in T_x N$, 
\[
0 = \omega( X_H(x) , v) = -dH(x)\cdot v.
\]
By definition, we also have $\omega(X_H(x), v) = 0$ for all $v \in \ker dH$. 
By hypothesis, $T_x M = \ker dH_x + T_x N$, so $\omega(X_H(x), v) = 0$ for all $v \in T_x M$,
hence $X_H(x) = 0$. 
\end{proof}

To conclude the proof, we recall that, by assumption, $N \pitchfork S_\lambda$ for every $S_\lambda \subset U$, so at each $x \in N \cap U$, we have $T_x N + \ker dH_x = T_x M$. 
By the construction of $F$, we have $dF_x = f'(H(x)) dH_x$ so 
$\ker dH_x \subset \ker dF_x$, and thus the hypotheses of the lemma are verified for $F$. 
It then follows that $X_F(x)$ either vanishes or points out of the isotropic leaf.

The remainder of the proof now proceeds as in \cite{Hofer_Zehnder_1994}. We
include it here for the convenience of the reader. Since $m(F) > {c}(U,N,\omega,
\sim)$, there exists a nonconstant leafwise chord $x(t)$
with return time $0 < T \leq 1$ which is a solution of the Hamiltonian system $\dot{x}(t) = X_F(x(t))$. 
Since $F = f(H)$, we have
\begin{equation*}
X_F(x) = f'(H(x))(X_H(x)) \ .
\end{equation*}
Also, note that, for a solution $x(t)$ of the Hamiltonian equation, $H(x(t)) = \lambda$ is constant in $t$, since 
\begin{equation*}
\frac{d}{dt}H(x(t))=dH(x(t))\cdot\dot{x}(t) = f'(H)\omega(X_H,X_H) = 0.
\end{equation*}
Since $x(t)$ is non-constant we must have 
\begin{equation*}
f'(H(x(t))) = f'(\lambda) \neq 0.
\end{equation*}
From the definition of $f$, we see that $\lambda \in (1-\epsilon,1-\frac{\epsilon}{2}) \cup (1+\frac{\epsilon}{2},1+\epsilon)$. Let $\tau \colonequals f'(\lambda)$. Reparametrizing, we define $y:\R \to S_\lambda$ by $y(t)\colonequals x(\frac{t}{\tau})$. This curve has period $\tau T$ and satisfies the equation
\begin{equation*}
\bar{y}(t) = \frac{1}{\tau}\bar{x}(t)=X_H(y(t)),
\end{equation*} 
and is therefore a solution of the original Hamiltonian equation on the energy surface $S_\lambda$. Since $\epsilon$ is arbitrary, we have shown that there exists a sequence $\lambda_j \to \alpha$ of energy levels such that there is a leafwise chord on each $S_j$. However, the same argument proves this for any $\lambda \in I$. Therefore, the theorem is proved. 
\end{proof}

\begin{rem*}

This theorem only guarantees the existence of leafwise chords near a given
energy level and says nothing about the energy level itself. However,
if we add the assumption that the return times $T_j$ of the solutions
$x_j(t)$ on each $S_{\lambda_j}$ are uniformly bounded, and that $S$
and each $S_{\lambda_j}$ are compact, then a standard Arzel\`{a}-Ascoli
argument together with Lemma \ref{lem:Transversality and admissibility} (which
prevents the resulting limit from being contained in a leaf) allows us to conclude:
\begin{prop}
Let $(M,\omega)$ be a symplectic manifold, $N\hookrightarrow M$ be a coisotropic submanifold. Let $H:M \to \R$ be a Hamiltonian function with Hamiltonian vector field $X_H$, and suppose there is an energy level $S_\alpha$ which is compact and such that $N\pitchfork S_\alpha$. Furthermore, let $\lambda_j \to \alpha$ and assume that the return times $T_j$ of the leafwise Hamiltonian chords $x_j(t)$ are bounded by some $\beta>0$ and that the $S_{\lambda_j}$ are compact. Then $S = S_\alpha$ admits a leafwise Hamiltonian chord which is a solution of the equation $\bar{x}(t) = X_H(x(t))$.
\qed
\end{prop}
\end{rem*}

Similarly, applying Lemma \ref{lem:Transversality and admissibility} 
to obtain compactness for non-trivial chords of bounded length, 
we may adapt many
results proving the existence of periodic orbits on energy surfaces
to our context of chords on coisotropic submanifolds. We finish by
stating two such results here on the existence of leafwise Hamiltonian
chords on energy surfaces transverse to a coisotropic submanifold $N$
of $(M, \omega)$. The proofs are modifications of the proofs of Theorems 3 and 4 in
\cite{Hofer_Zehnder_1994}*{Chapter 4}, using Lemma \ref{lem:Transversality
and admissibility} and the same strategy as in the proof of Theorem
\ref{thm:Almost existence}. We omit them here.

Before stating the next theorem, we recall two definitions from
\cite{Hofer_Zehnder_1994}. First, \emph{a parametrized family of
hypersurfaces based on $S$} is a diffeomorphism $\psi:S \times I \to U
\subset M$, where $I$ is an open interval containing $0$, $U$ is bounded,
and $\psi(x,0) = x$ for all $x\in S$.

Now suppose that each hypersurface $S_\epsilon$ in a parametrized family of hypersurfaces based on $S$ bound a symplectic manifold $U_\epsilon$. We say that $S_\epsilon$ is \emph{of ${c}$-Lipschitz type} if there are positive constants $L$ and $a$ such that 
\begin{equation*}
{c}(U_\epsilon,N,\omega, \sim) < {c}(U_{\epsilon^*},N,\omega, \sim) + L(\epsilon - \epsilon^*)
\end{equation*}
for all $\epsilon^* < \epsilon < \epsilon^* + L(\epsilon - \epsilon^*)$.

When $S$ is a hypersurface as above, and $N$ is a coisotropic submanifold such
that $S$ and $N$ intersect transversally, we write 
$\mathcal{C}(S, N)$ to denote the set of leafwise Hamiltonian chords on $S$
for any Hamiltonian that has $S$ as a regular level set.

\begin{thm}
Let $N \hookrightarrow M$ be a coisotropic submanifold of $(M,\omega)$, and
suppose that $c(M,N,\omega, \sim) < \infty$. Let $S \hookrightarrow M$ be a
compact hypersurface that intersects $N$ transversally and which bounds a
compact symplectic submanifold of $M$. If $S$ is of $c_0$-Lipschitz type, then
$\mathcal{C}(S, N) \neq \emptyset$. \qed
\end{thm}

\begin{thm}
Let $N \hookrightarrow M$ be a coisotropic submanifold of $(M,\omega)$, and suppose that $c(M,N,\omega,\sim) < \infty$. Suppose the compact hypersurface $S \hookrightarrow M$ bounds a compact symplectic manifold. Let $S_\epsilon$, with $\epsilon \in I$ be a parametrized family of hypersurfaces modelled on $S$, with $S_\epsilon$ transverse to $N$ for each $\epsilon \in I$. Then
\begin{equation*}
\mu \left\{ \epsilon \in I \, \vert \, \mathcal{C}(S_\epsilon, N) \neq \emptyset
\right\} = \mu(I),
\end{equation*}
where $\mu$ denotes the Lebesgue measure on $\R$.
\qed
\end{thm}

\acknowledgement{ 
The second author would like to thank the Laboratoire Jean Leray of the
Universit\'e de Nantes and the D\'epartement de Math\'ematiques
of the Universit\'e Libre de Bruxelles for their hospitality and the pleasant 
atmosphere
during his visits to work on this project, and the Institut Math\'ematiques de 
Toulouse at the Universit\'e Paul Sabatier for the invitation to give a seminar 
talk on an early version of these results. 

Both authors are grateful to 
Sobhan Seyfaddini, R\'emi Leclercq, 
Vincent Humili\`ere, Matthew Strom Borman, Leonid Polterovich, Felix Schlenk, 
Kaoru Ono, Yoshihiro Sugimoto, 
and Emmy Murphy  
for helpful feedback and interesting and useful discussions. 

We also thank the referee for very detailed feedback for improvement.
}

\bibliographystyle{amsplain}

\begin{bibdiv}
\begin{biblist}

\bib{Albers_Frauenfelder_2010_TandA}{article}{
      author={Albers, Peter},
      author={Frauenfelder, Urs},
       title={Leaf-wise intersections and {R}abinowitz {F}loer homology},
        date={2010},
        ISSN={1793-5253},
     journal={J. Topol. Anal.},
      volume={2},
      number={1},
       pages={77\ndash 98},
         url={https://doi-org.umiss.idm.oclc.org/10.1142/S1793525310000276},
      review={\MR{2646990}},
}

\bib{Albers_Momin_2010}{article}{
      author={Albers, Peter},
      author={Momin, Al},
       title={Cup-length estimates for leaf-wise intersections},
        date={2010},
        ISSN={0305-0041},
     journal={Math. Proc. Cambridge Philos. Soc.},
      volume={149},
      number={3},
       pages={539\ndash 551},
         url={https://doi-org.umiss.idm.oclc.org/10.1017/S0305004110000435},
      review={\MR{2726731}},
}

\bib{Barraud_Cornea_2007}{article}{
      author={Barraud, Jean-Fran\c{c}ois},
      author={Cornea, Octav},
       title={Lagrangian intersections and the {S}erre spectral sequence},
        date={2007},
        ISSN={0003-486X},
     journal={Ann. of Math. (2)},
      volume={166},
      number={3},
       pages={657\ndash 722},
         url={https://doi.org/10.4007/annals.2007.166.657},
      review={\MR{2373371}},
}

\bib{biran_cornea_2008}{incollection}{
      author={Biran, Paul},
      author={Cornea, Octav},
       title={A {L}agrangian quantum homology},
        date={2009},
   booktitle={New perspectives and challenges in symplectic field theory},
      series={CRM Proc. Lecture Notes},
      volume={49},
   publisher={Amer. Math. Soc., Providence, RI},
       pages={1\ndash 44},
      review={\MR{2555932}},
}

\bib{Biran_Cornea_2009}{article}{
      author={Biran, Paul},
      author={Cornea, Octav},
       title={Rigidity and uniruling for {L}agrangian submanifolds},
        date={2009},
        ISSN={1465-3060},
     journal={Geom. Topol.},
      volume={13},
      number={5},
       pages={2881\ndash 2989},
         url={https://doi.org/10.2140/gt.2009.13.2881},
      review={\MR{2546618}},
}

\bib{Borman_McLean_2014}{article}{
      author={Borman, Matthew~Strom},
      author={McLean, Mark},
       title={Bounding {L}agrangian widths via geodesic paths},
        date={2014},
        ISSN={0010-437X},
     journal={Compos. Math.},
      volume={150},
      number={12},
       pages={2143\ndash 2183},
         url={http://dx.doi.org/10.1112/S0010437X14007465},
      review={\MR{3292298}},
}

\bib{Buhovsky_2010}{article}{
      author={Buhovsky, Lev},
       title={A maximal relative symplectic packing construction},
        date={2010},
        ISSN={1527-5256},
     journal={J. Symplectic Geom.},
      volume={8},
      number={1},
       pages={67\ndash 72},
         url={http://projecteuclid.org/euclid.jsg/1271166375},
      review={\MR{2609629}},
}

\bib{Casals_Spacil_2016}{article}{
      author={Casals, Roger},
      author={Sp{\'a}{\v c}il, Old{\v r}ich},
       title={Chern-{W}eil theory and the group of strict contactomorphisms},
        date={2016},
        ISSN={1793-5253},
     journal={J. Topol. Anal.},
      volume={8},
      number={1},
       pages={59\ndash 87},
         url={http://dx.doi.org.umiss.idm.oclc.org/10.1142/S1793525316500035},
      review={\MR{3463246}},
}

\bib{Cieliebak_Hofer_Latschev_Schlenk_2007}{incollection}{
      author={Cieliebak, Kai},
      author={Hofer, Helmut},
      author={Latschev, Janko},
      author={Schlenk, Felix},
       title={Quantitative symplectic geometry},
        date={2007},
   booktitle={Dynamics, ergodic theory, and geometry},
      series={Math. Sci. Res. Inst. Publ.},
      volume={54},
   publisher={Cambridge Univ. Press, Cambridge},
       pages={1\ndash 44},
         url={https://doi.org/10.1017/CBO9780511755187.002},
      review={\MR{2369441}},
}

\bib{Deimling_1985}{book}{
      author={Deimling, Klaus},
       title={Nonlinear functional analysis},
   publisher={Springer-Verlag, Berlin},
        date={1985},
        ISBN={3-540-13928-1},
         url={https://doi.org/10.1007/978-3-662-00547-7},
      review={\MR{787404}},
}

\bib{Rizell_2013}{article}{
      author={Dimitroglou~Rizell, Georgios},
       title={Exact {L}agrangian caps and non-uniruled {L}agrangian
  submanifolds},
        date={2015},
        ISSN={0004-2080},
     journal={Ark. Mat.},
      volume={53},
      number={1},
       pages={37\ndash 64},
         url={https://doi-org.umiss.idm.oclc.org/10.1007/s11512-014-0202-y},
      review={\MR{3319613}},
}

\bib{Dragnev_2008}{article}{
      author={Dragnev, Dragomir~L.},
       title={Symplectic rigidity, symplectic fixed points, and global
  perturbations of {H}amiltonian systems},
        date={2008},
        ISSN={0010-3640},
     journal={Comm. Pure Appl. Math.},
      volume={61},
      number={3},
       pages={346\ndash 370},
         url={https://doi-org.umiss.idm.oclc.org/10.1002/cpa.20203},
      review={\MR{2376845}},
}

\bib{Ekeland_Hofer_1989}{article}{
      author={Ekeland, Ivar},
      author={Hofer, Helmut},
       title={Symplectic topology and {H}amiltonian dynamics},
        date={1989},
        ISSN={0025-5874},
     journal={Math. Z.},
      volume={200},
      number={3},
       pages={355\ndash 378},
         url={http://dx.doi.org/10.1007/BF01215653},
      review={\MR{978597 (90a:58046)}},
}

\bib{Ekholm_etal_2013}{article}{
      author={Ekholm, Tobias},
      author={Eliashberg, Yakov},
      author={Murphy, Emmy},
      author={Smith, Ivan},
       title={Constructing exact {L}agrangian immersions with few double
  points},
        date={2013},
        ISSN={1016-443X},
     journal={Geom. Funct. Anal.},
      volume={23},
      number={6},
       pages={1772\ndash 1803},
         url={http://dx.doi.org.umiss.idm.oclc.org/10.1007/s00039-013-0243-6},
      review={\MR{3132903}},
}

\bib{Ginzburg_2007}{article}{
      author={Ginzburg, Viktor~L.},
       title={Coisotropic intersections},
        date={2007},
        ISSN={0012-7094},
     journal={Duke Math. J.},
      volume={140},
      number={1},
       pages={111\ndash 163},
  url={https://doi-org.umiss.idm.oclc.org/10.1215/S0012-7094-07-14014-6},
      review={\MR{2355069}},
}

\bib{Gromov_1985}{article}{
      author={Gromov, Mikhail},
       title={Pseudo holomorphic curves in symplectic manifolds},
        date={1985},
        ISSN={0020-9910},
     journal={Invent. Math.},
      volume={82},
      number={2},
       pages={307\ndash 347},
         url={https://doi.org/10.1007/BF01388806},
      review={\MR{809718}},
}

\bib{Gurel_2010}{article}{
      author={G\"{u}rel, Ba\c{s}ak~Zehra},
       title={Leafwise coisotropic intersections},
        date={2010},
        ISSN={1073-7928},
     journal={Int. Math. Res. Not. IMRN},
      number={5},
       pages={914\ndash 931},
         url={https://doi-org.umiss.idm.oclc.org/10.1093/imrn/rnp164},
      review={\MR{2595016}},
}

\bib{Hofer_1990}{article}{
      author={Hofer, Helmut},
       title={On the topological properties of symplectic maps},
        date={1990},
        ISSN={0308-2105},
     journal={Proc. Roy. Soc. Edinburgh Sect. A},
      volume={115},
      number={1-2},
       pages={25\ndash 38},
         url={https://doi-org.umiss.idm.oclc.org/10.1017/S0308210500024549},
      review={\MR{1059642}},
}

\bib{Hofer_Zehnder_1994}{book}{
      author={Hofer, Helmut},
      author={Zehnder, Eduard},
       title={Symplectic invariants and {H}amiltonian dynamics},
      series={Birkh\"{a}user Advanced Texts: Basler Lehrb\"{u}cher.
  [Birkh\"{a}user Advanced Texts: Basel Textbooks]},
   publisher={Birkh\"{a}user Verlag, Basel},
        date={1994},
        ISBN={3-7643-5066-0},
         url={https://doi.org/10.1007/978-3-0348-8540-9},
      review={\MR{1306732}},
}

\bib{Kang_2013}{article}{
      author={Kang, Jungsoo},
       title={Generalized {R}abinowitz {F}loer homology and coisotropic
  intersections},
        date={2013},
        ISSN={1073-7928},
     journal={Int. Math. Res. Not. IMRN},
      number={10},
       pages={2271\ndash 2322},
         url={https://doi-org.umiss.idm.oclc.org/10.1093/imrn/rns113},
      review={\MR{3061940}},
}

\bib{Katok_1973}{article}{
      author={Katok, Anatole~B.},
       title={Ergodic perturbations of degenerate integrable {H}amiltonian
  systems},
        date={1973},
        ISSN={0373-2436},
     journal={Izv. Akad. Nauk SSSR Ser. Mat.},
      volume={37},
       pages={539\ndash 576},
      review={\MR{0331425}},
}

\bib{Moser_1978}{article}{
      author={Moser, J{\" u}rgen},
       title={A fixed point theorem in symplectic geometry},
        date={1978},
        ISSN={0001-5962},
     journal={Acta Math.},
      volume={141},
      number={1--2},
       pages={17\ndash 34},
         url={https://doi-org.umiss.idm.oclc.org/10.1007/BF02545741},
      review={\MR{0478228}},
}

\bib{Rieser_2014}{article}{
      author={Rieser, Antonio},
       title={Lagrangian blow-ups, blow-downs, and applications to real
  packing},
        date={2014},
        ISSN={1527-5256},
     journal={J. Symplectic Geom.},
      volume={12},
      number={4},
       pages={725\ndash 789},
         url={http://projecteuclid.org/euclid.jsg/1433196062},
      review={\MR{3333028}},
}

\bib{Schlenk_2005}{article}{
      author={Schlenk, Felix},
       title={Packing symplectic manifolds by hand},
        date={2005},
        ISSN={1527-5256},
     journal={J. Symplectic Geom.},
      volume={3},
      number={3},
       pages={313\ndash 340},
         url={http://projecteuclid.org/euclid.jsg/1144954876},
      review={\MR{2198779}},
}

\bib{Usher_2011}{article}{
      author={Usher, Michael},
       title={Boundary depth in {F}loer theory and its applications to
  {H}amiltonian dynamics and coisotropic submanifolds},
        date={2011},
        ISSN={0021-2172},
     journal={Israel J. Math.},
      volume={184},
       pages={1\ndash 57},
         url={https://doi.org/10.1007/s11856-011-0058-9},
      review={\MR{2823968}},
}

\bib{Zehmisch_2012a}{article}{
      author={Zehmisch, Kai},
       title={The codisc radius capacity},
        date={2013},
        ISSN={1935-9179},
     journal={Electron. Res. Announc. Math. Sci.},
      volume={20},
       pages={77\ndash 96},
      review={\MR{3117101}},
}

\bib{Zehmisch_2012b}{article}{
      author={Zehmisch, Kai},
       title={Lagrangian non-squeezing and a geometric inequality},
        date={2014},
        ISSN={0025-5874},
     journal={Math. Z.},
      volume={277},
      number={1-2},
       pages={285\ndash 291},
         url={https://doi-org.umiss.idm.oclc.org/10.1007/s00209-013-1254-6},
      review={\MR{3205772}},
}

\bib{Zeidler_1986}{book}{
      author={Zeidler, Eberhard},
       title={Nonlinear functional analysis and its applications. {I}:
  Fixed-point theorems},
   publisher={Springer-Verlag, New York},
        date={1986},
        ISBN={0-387-90914-1},
         url={https://doi.org/10.1007/978-1-4612-4838-5},
        note={Translated from the German by Peter R. Wadsack},
      review={\MR{816732}},
}

\bib{Ziltener_2010}{article}{
      author={Ziltener, Fabian},
       title={Coisotropic submanifolds, leaf-wise fixed points, and
  presymplectic embeddings},
        date={2010},
        ISSN={1527-5256},
     journal={J. Symplectic Geom.},
      volume={8},
      number={1},
       pages={95\ndash 118},
  url={http://projecteuclid.org.umiss.idm.oclc.org/euclid.jsg/1271166377},
      review={\MR{2609631}},
}

\end{biblist}
\end{bibdiv}


\end{document}